\documentclass [12pt]{amsart}
\usepackage[utf8]{inputenc}

\usepackage{amsmath,amssymb,amsthm,amsfonts}
\usepackage{mathtools}%

\usepackage[main=english,french]{babel}%
\usepackage{bbm}%
\usepackage{mathrsfs}%

\usepackage{tikz,graphicx,color}
\usepackage{tikz-cd}%
\usepackage{tikz-3dplot}%

\usetikzlibrary{calc}%
\usetikzlibrary{arrows}%
\usetikzlibrary{shapes}%
\usetikzlibrary{patterns}%
\usetikzlibrary{positioning}%
\usetikzlibrary{arrows.meta}
\usetikzlibrary{decorations.markings}

\usepackage[margin=1in]{geometry}

\usepackage{soul}
\usepackage{accents}
\usepackage{enumitem}
\usepackage{letltxmacro}
\usepackage{thmtools,etoolbox}

\usepackage{nameref,hyperref}
\usepackage[capitalize]{cleveref}

\usepackage{tikz-cd}
\usepackage{stmaryrd}

\usepackage{xspace}

\usepackage{comment}

\usepackage{stmaryrd}

\usepackage{graphicx}
\usepackage{amsfonts}
\usepackage{amsmath}
\usepackage{amssymb}
\usepackage{amsthm,color}
\usepackage{epsfig,latexsym}
\usepackage{booktabs}
\usepackage{advdate}
\usepackage{tikz, tikz-cd}
\usepackage{physics}
\usepackage{blkarray, array}
\newcolumntype{L}{>{$}l<{$}}

\newtheorem{Thm}{Theorem}
\newtheorem*{Thm*}{Theorem}

\newtheorem{Cor}{Corollary}
\newtheorem{Lem}{Lemma}
\newtheorem{Prop}{Proposition}
\newtheorem*{Prop*}{Proposition}

\newtheorem{Def}{Definition}

\newtheorem{Exa}{Example}

\newcommand{\C}{\mathbb{C}}
\newcommand{\Z}{\mathbb{Z}}

\def\rin{\rotatebox[origin=c]{90}{$\in$}}

\def\+{\includegraphics[scale=0.3]{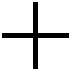}}
\def\-{\includegraphics[scale=0.3]{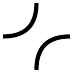}}

\newcommand{\ra}{\ensuremath{\xrightarrow}}

\newcommand\rd[1]{{\color{red}#1}}
\newcommand\teal[1]{{\color{teal}#1}}

\newcommand{\mf}{\ensuremath{\mathfrak}}

\DeclareMathOperator{\codim}{codim}

\newcommand\minip[3]{\noindent
  \begin{minipage}{#1\linewidth}
    \hskip #2
    #3
  \end{minipage}
}

\begin{document}


\date{\today}

\author{Joseph Fluegemann}
\title{Smooth Points on Positroid Varieties}

\email{jkf68@cornell.edu}
\date{\today}

\maketitle

\begin{abstract}

In the Grassmannian $Gr_\C(k,n)$ we have positroid varieties $\Pi_f$, each indexed by a bounded affine permutation $f$ and containing torus-fixed points $\lambda \in \Pi_f$. In this paper we consider the partially ordered set consisting of quadruples $(k,n,\Pi_f,\lambda)$ (or \textit{(positroid) pairs} $(\Pi_f,\lambda)$ for short). The partial order is the ordering given by the covering relation $\lessdot$ where $(\Pi_f',\lambda') \lessdot (\Pi_f,\lambda)$ if $\Pi_f'$ is obtained by $\Pi_f$ by \textit{deletion} or \textit{contraction.} Using the results of Snider \cite{Sni}, we know that positroid varieties can be studied in a neighborhood of each of these points by \textit{affine pipe dreams.} Our main theorem provides a quick test of when a positroid variety is smooth at one of these given points. It is sufficient to test smoothness of a positroid variety by using the main result to test smoothness at each of these points. These results can also be applied to the question of whether Schubert varieties in flag manifolds are smooth at points given by 321-avoiding permutations, as studied in \cite{GraKre}. We have a secondary result, which describes the minimal singular positroid pairs in our ordering - these are the positroid pairs where any deletion or contraction causes it to become smooth.  
\end{abstract}

\tableofcontents

\section{Introduction}
 \subsection{Motivation from Schubert varieties}

In order to provide context of our smoothness problem for those familiar with Schubert varieties, we note both the similarities and differences compared to the case of smoothness for Schubert varieties. In flag manifolds, the smooth Schubert varieties were classified via a pattern avoidance technique by Lakshmibai and Sandiya \cite{LS90}. Note that if we just restrict to the Grassmannian, it turns out that the smooth Schubert varieties are essentially trivial - for any $Gr(k,n)$, the smooth Schubert varieties are just the ones that are subGrassmannians (isomorphic to $Gr(k',n')$ for $k'\leq k, n'\leq n$). We emphasize that in this paper, we are not classifying smooth positroid varieties, as was done in \cite{BW22}. Rather, we are only considering smoothness at the fixed points, one fixed point at a time. 

Inside a given Schubert variety $X_w$, for any fixed point $v$, the Schubert subvariety $X_v$ is the $B_-$-orbit closure of the point $v$, and the singular locus for any given Schubert variety is invariant under the action of $B_-$, so the singular locus must consist of a union of Schubert subvarieties. Within $X_w$, these Schubert varieties $X_v$ on which $X_w$ is singular are obviously closed under going up in the Bruhat order, since if $X_w$ is singular along $X_v$ then $X_w$ will also be singular on any subset of $X_v$. Then, the natural question of what are the maximal (that is, minimal in Bruhat order) such Schubert varieties $X_v$  was answered by the 321-hexagon-avoiding permutations described by Billey and Warrington \cite{BilWar}. 

Another question of checking whether $X_w$ is singular at a $T$-fixed point $v$ was answered by Dale Peterson: consider $ s= \# \{ v r_\alpha \geq w \}$ where $\alpha$ is over positive roots. Since each $v r_\alpha \geq w$ produces a linearly independent vector in the tangent space, $s \leq \dim T_v (X_w)$. Thus is $s > \dim (X_w)$ then the tangent space is too large so $X_w$ is singular at $v$. Conversely if the group is simply-laced, as it is in the case when the group is a torus, then $s = \dim T_v (X_w)$ implies that $X_w$ is smooth at $v$.

We answer the corresponding questions for positroid varieties, both the question of whether a positroid variety $\Pi_f$ is smooth at a given fixed point $\lambda$, and also what the minimal singular positroid varieties are, under an ordering given by deletion and contraction.

\subsection{The Results}\label{sub:theresult}

First, we describe our poset (with more details given in Subsection \ref{subsection:delcontr}). We make use of the following two subGrassmannians inside $Gr(k,n)$:
\\

\noindent $\Pi_{del,i} := \{V\in Gr(k,n) \mid  V\leq Proj_i(\C^n)\cong \C^{n-1} \} \cong Gr(k,n-1)$ 
\\

\noindent $\Pi_{contr,i} := \{V\in Gr(k,n) \mid  V\geq \C_i \} \cong Gr(k-1,n-1)$
\\

The $T$-fixed points on $Gr(k,n)$ are given by the span of $k$ coordinates lines, which we have denoted by $\lambda$, so there are $n\choose k$ such fixed points. Given such a $T$-fixed point $\lambda$ in $Gr(k,n)$ and a coordinate $x_i$, then either that fixed point $\lambda$ uses $x_i$ or it does not. This can be rephrased as: for each $T$-fixed point $\lambda$ and each $i\in \{1,...,n\}$, either $\lambda \in \Pi_{del,i}$ or $\lambda \in \Pi_{contr,i}$, but not both. We can intersect a positroid variety $\Pi_f$ with one of these two subGrassmannians, and what we obtain is a smaller positroid variety $\Pi_{f'}$ and a point $\lambda'$ on it. Call this $\Pi_{f'}$ the \textbf{deletion} or \textbf{contraction} respectively. In other words: 
\\

(1) If we have a $T$-fixed point $\lambda$ on a positroid variety $\Pi_f$ and we have one of the coordinates $x_i$, then we can either delete or contract $x_i$, and we will get to another point $\lambda'$ on a smaller positroid variety $\Pi_{f'}$.
\\

(2) Furthermore, if the point $\lambda$ we started out with is smooth on $\Pi_f$, then the point $\lambda'$ after deletion or contraction will again be smooth on $\Pi_f \cap \Pi_{del,i}$ or $\Pi_f\cap \Pi_{contr,i}$ (this will be proved as Proposition \ref{prop:delsmooth}). 
\\

Because of (1), we can make an (infinite) poset on the set of quadruples $(k,n,\Pi_f,\lambda)$ with deletion and contraction as the covering relations. By (2), the set of singular quadruples is closed under going upwards. We can visualize the poset as:

 \begin{figure}[htbp]
	\includegraphics[scale=0.5,clip=true]{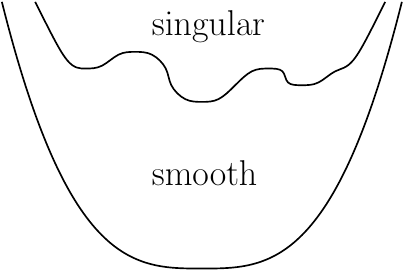}
\end{figure}

Given this poset, we can consider the points along the boundary (the minimal singular pairs):

 \begin{figure}[htbp]
	\includegraphics[scale=0.5,clip=true]{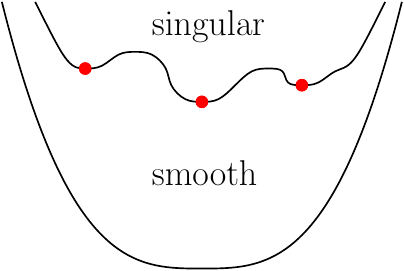}
\end{figure}

Since $n$ decreases as one goes down in the ordering, we can only go down in the order a finite number of steps from any given point. By this finiteness of going down, a pair is singular if and only if it is greater than one of these minimal singular pairs in the ordering. We give these minimal singular pairs a name:

\begin{Def}[Atomic Positroid Pair]
A positroid variety $\Pi_f$ is \textbf{atomic} on an open patch $U_\lambda$ if it is singular at $\lambda$, but any deletion or contraction (of any column) causes this positroid variety to be smooth (or empty) at $\lambda$. We call $(\Pi_f, \lambda)$ an \textbf{atomic positroid pair}.
\end{Def}

\vspace{1cm}

Next, we recall that an affine pipe dream \textbf{shape} is the boundary of an affine pipe dream \cite{Sni}. More details can be found in Section \ref{section:affinepipe}. A quick example is provided by the following. 

\begin{figure}[htbp]\includegraphics[scale=0.7,clip=true]{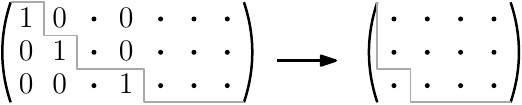}\end{figure}

Here, we are in $Gr(3,7)$ and $\lambda=\{1,2,4\}$ (these are the full rank columns). We ``collapse" these columns, such that they are now depicted by the vertical segments in the \textbf{distinguished path} (staircase line) running from the NW to SE corners. The affine pipe dream's shape is an infinite strip obtained by taking this distinguished path and placing an exact copy directly above it and another copy to the right. This continues forever in the NW and SE directions as depicted in the figure below (note that the numbers within the pipe dream shape can be ignored for now). 

\begin{figure}[htbp]
\includegraphics[scale=0.7,clip=true]{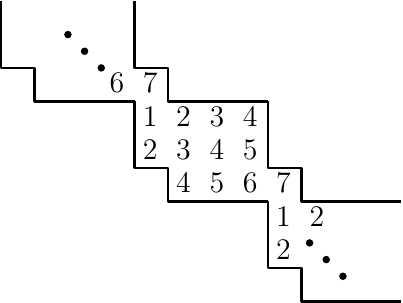}
\end{figure}

\newpage

If the numbers in $\lambda$ are consecutive, say $\lambda=\{i,i+1,i+2\}$, then the shape obtained will be an infinite string of rectangles touching in the northwest/southeast corners, as depicted in the following:

 \begin{figure}[htbp]
	\includegraphics[scale=0.6,clip=true]{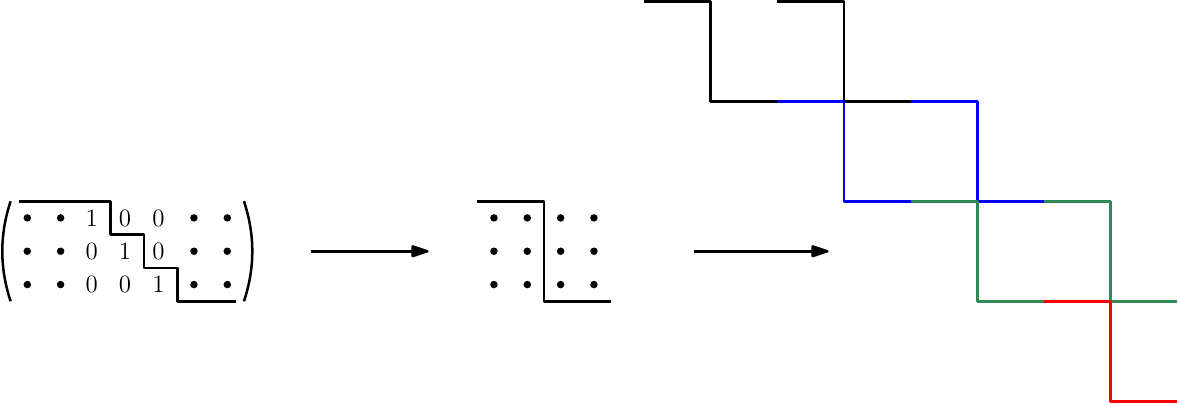}
\end{figure}

We remark that Grassmannian duality: $Gr(k,n)\sim Gr(n-k,n)$ reflects the pipe dreams shapes across the line $x=-y$ (and this also switches deletion and contraction).
\\

Additionally, this is what the \textbf{maximal rectangles} look like in our affine shapes:

 \begin{figure}[htbp]
	\includegraphics[scale=0.5,clip=true]{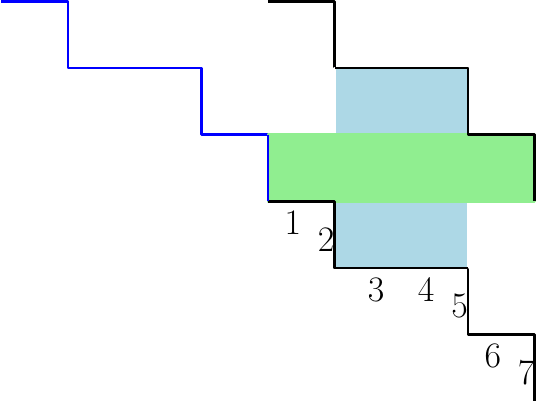}
\end{figure}

One goal we have is to quickly determine whether a positroid pair is singular at a given point. Our main method uses the following Theorem \ref{thm:smoothone}:

\begin{Thm*}
A positroid variety $\Pi_f$ is smooth at the point $\lambda$ (it meets $\lambda$ and is smooth there) if and only if there is a single affine pipe dream representative \cite{Sni} of the affine permutation $f$ on the open set $U_\lambda$. Specifically $\Pi_f$ meets the point $\lambda$ if and only if there exists (at least) one affine pipe dream, and it is smooth if and only if that affine pipe dream is unique. Thus, we can take $\Pi_f$ and look at pipe dreams for all choices of $\lambda$ to test for smoothness. 
\end{Thm*}

By applying the previous theorem, we can prove the Main Theorem \ref{thm:mainthm}. To state this theorem, we first need two definitions:

\begin{Def}
We have a \textbf{partition} in the northwest of a rectangle if the maximal connected set of crosses containing the cross in the northwest corner (it can be empty) is of the following type of shape: if each successive row as we go down consists of a consecutive crosses coming from the left edge of the rectangle such that the number of such crosses is nonincreasing (equivalently, if we take the axes as going right and down, given any cross at $(a,b)$, all $(i,j)$ with $i\leq a$ and $j\leq b$ must be crosses). A partition in the southeast is the same, but reflected across the southwest-northeast axis. Here is an example of a northwest partition:

 \begin{figure}[htbp]
	\includegraphics[scale=0.4,clip=true]{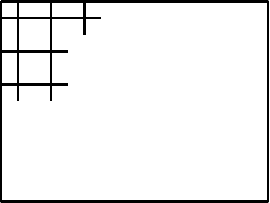}
\end{figure}

\end{Def}

\begin{Def}
A pipe dream in a rectangular shape \textbf{reduces to SE/NW partitions} if, after deleting all entire rows and columns of crosses, the result is a partition in the southeast and northwest corners. 

This is equivalent to the condition that all crosses must be contained in: (1) entire rows of crosses, (2) entire columns of crosses, (3) a partition in the northwest, (4) a partition in the southeast.
\end{Def}

\begin{Exa}
The following is an example of a rectangular pipe dream that reduces to SE/NW partitions:
 \begin{figure}[htbp]
	\includegraphics[scale=0.5,clip=true]{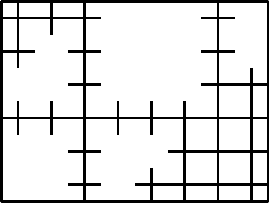}
\end{figure}
\end{Exa}

\begin{Thm*}[Main Theorem]
If each maximal rectangle in the pipe dream for $\Pi_f$ on $U_\lambda$ reduces to SE/NW partitions, then $\Pi_f$ is smooth at $\lambda$; otherwise, it is singular. 
\end{Thm*}

Thus, if we check all maximal rectangles and, for each one, after removing all entire rows and columns of crosses, see that we are left with a partition in the northwest and one in the southeast, then our positroid variety is smooth at the given point; otherwise, it is singular. In other words, within each maximal rectangle, the only crosses that exist are contained in: (1) entire rows of crosses, (2) entire columns of crosses, (3) a partition in the northwest, (4) a partition in the southeast; in any given maximal rectangle, there can be many possible instances of (1) and (2), but only a single partition (3) and a single partition (4). 
\\

Another method for determining from the notion of top and bottom pipe dreams. From \cite{BB93}, we know that in the matrix Schubert variety case, the set of all pipe dreams for a given permutation has two polar opposites: the top and bottom pipe dreams. In Appendix B, we will show that this is also true in the affine pipe dream case. Thus, a second way to determine whether a pair is smooth is to find the top and bottom pipe dreams and see if they match. 
\\

Finally, although not an algorithmic method like the previous two, we also have a condition for whether a positroid pair $(\Pi_f,\lambda)$ is smooth or singular using atomic positroid pairs, which we defined above; this essentially follows from the definition of the partial order, and is the content of Proposition \ref{prop:atom=sing}:

\begin{Prop*}
A point $\lambda$ being singular on a positroid variety $\Pi_f$  is equivalent to the pipe dream for $\Pi_f$ on $U_\lambda$ being able to reach an atomic configuration via a series of deletions and contractions. 
\end{Prop*}

We want to classify these atomic positroid pairs. This is given in Theorem \ref{thm:atomicpd}:
\\

\begin{Thm*}

The atomic positroid pairs have affine pipe dreams that look like the following:

 \begin{figure}[htbp]
	\includegraphics[scale=0.6,clip=true]{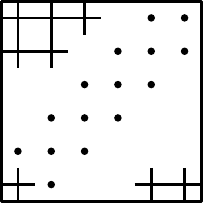}
\end{figure}

Specifically, the atomic positroid pairs:

(1) Have pipe dreams in a square shape: so they are in $Gr(k,2k)$ and on points where $\lambda$ consists of $k$ consecutive columns. 

(2) In this pipe dream, there is a single cross along the southwest to northeast (longest) diagonal, with the diagonals directly adjacent to this longest diagonal being free of any crosses.

(3) Besides this one cross, there are partitions of crosses in the northwest and southeast corners, and no others.

\end{Thm*}

Finally, we remark that the techniques in the proof of the Main Theorem can be easily extended to the case of Schubert varieties in flag manifolds at points given by 321-avoiding permutations (Prop \ref{prop:flagscorollary}):

\begin{Prop*}
Let $Fl(n)$ denote the variety of flags in $\C^n$. Let $w\in S_n$ denote a permutation on $n$ elements so that $X_w$ denotes a Schubert variety in this flag manifold, and let $v\in S_n$ be 321-avoiding with $v\geq w$ in Bruhat order. Then $X_w$ is smooth at $v$,  if and only if there exists a pipe dream for $w$ inside $v$'s skew partition (defined above Prop \ref{prop:nonaffineword}) such that all maximal rectangles reduce to NW/SE partitions. (In this case, the pipe dream will be unique.) 
\end{Prop*}

\section{Background Information}

 \subsection{Positroid Varieties}

Our background for the following results on positroid varieties is \cite{KLS}. 

\begin{Def}[Bounded Affine Permutation]
Fix $n\in \mathbb{N}$. A permutation $g:\Z \to \Z$ is called \textbf{affine} if it satisfies the periodicity condition $g(i+n)=g(i)+n$, $\forall i \in \Z$ .

If the affine permutation also satisfies $i\leq g(i)\leq i+n$ for all $i$, it is called \textbf{bounded}. We denote the set of bounded affine permutations of period $n$ and average jump $avg(g(i)-i)=k$ by $Bound(k,n)$. 

Given that an affine permutation repeats with period $n$, it is sufficient to give a permutation by stating its action on any $n$ consecutive numbers. This is called \textbf{window notation}. 
\end{Def}

\begin{Exa}
In window notation, $g=(4,6,3,7,11,8)$ defines a bounded affine permutation of period $6$ sending $g(1)=4, g(2)=3,...,g(6)=8$ and repeating, so $g(7)=g(1)+6=10$, etc.
\end{Exa}

The information contained in a bounded affine permutation $f$ can be given a different, but equivalent, labelling called its \textit{siteswap}, denoted $Site(k,n)$.

\begin{Def}[Siteswap]
The siteswap $f\in Site(k,n)$ corresponding to a bounded affine permutation $g\in Bound(k,n)$ is $f(i):=g(i)-i, 1\leq i \leq n$. Notice that a siteswap $f(i)$ is only defined for $i\in[1,n]$.
\end{Def}

\begin{Exa}
The siteswap corresponded to $g=(4,6,3,7,11,8)$ is  $f=(3,4,0,3,6,2)$.
\end{Exa}

\emph{Remark}: Siteswaps or bounded affine permutations represent (one-handed, bounded) juggling patterns. Namely, the affine permutation can be considered as a permutation taking the time at which a ball is thrown to the time at which that ball lands (so the siteswap gives the number of seconds after a ball is thrown that it lands). 

Let $[i,j] := \{i, i+1,...,j\} \pmod n$, so $\abs{[i,j]}=\begin{cases} j-i+1 & j\geq i \\ (n-i+1)+j & i>j\end{cases}$. 
\\

Define $a_{g,i,j}:=\#\{[i,j]\backslash \{g(i+\mathbb{N}) \} \cap \{\leq j\} \}$. This can be visualized by drawing a permutation matrix for the bounded affine permutation $g$, as in the figures in Appendix A. For say $g(\tilde{i})=\tilde{j}$, the $\tilde{i}$ runs vertically downwards in the figure below, while $\tilde{j}$ runs horizontally to the right. 

We have: $rank_g[i,j]=\abs{[i,j]} - a_{g,i,j} = (j-i+1) - a_{g,i,j}$. 

Now, we can define the positroid variety corresponding to a siteswap $f$ (we apply $\bmod n$ to the definitions using $g$ above):
\\

\begin{Def}[Closed and Open Positroid Varieties]
\[\Pi_f := GL_k \setminus \{M\subseteq M_{k,n} \mid rankM=k \text{ and } rank_M([i,j]) \leq rank_f[i,j], i\leq j \leq i+n\} \]
\[\Pi_f^0 := GL_k \setminus \{M\subseteq M_{k,n} \mid rankM=k \text{ and } rank_M([i,j]) = rank_f[i,j], i\leq j \leq i+n\} \]
where $rank_M([i,j])$ is defined cyclically.
\end{Def}

The open positroid varieties $\Pi_f^0$ are open inside their closures (locally closed), and are all smooth.
\\

 \subsection{Deletion and Contraction}\label{subsection:delcontr}

The Grassmannian is represented by full-rank $k$-by-$n$ matrices (up to $GL(k)$ action on the left): 

\[\{M\}=
\left \{ \begin{blockarray}{cccccc}
 &  &  & n &  &  \\
\begin{block}{c(ccccc)}
& & & & & \\
k & & & * & & \\
& & & & & \\
\end{block}
\end{blockarray} \hspace{0.1cm}
\right \} 
 \]

Then there is a circle action on the right via matrix multiplication by 
{\footnotesize\[S_i=\mqty(\dmat{1,1,\ddots,1,z,1,\ddots,1})\]} 
and this scales the $i$-th column of $M$ by $z$. 
 \\
 
There is a ``Pascal recurrence" on Grassmannians:

\begin{Prop}

Let $S_i$ act on $\{M\}=Gr(k,n)$ on the right. 

Let

\[\Pi_{contr,i} := rowspan \left \{ \mqty(&&&0&&&\\  &*&&0&&*&\\ 0&\cdots&0&1&0&\cdots&0) \right \} \cong Gr(k-1,n-1) \] 

and

\[\Pi_{del,i} := rowspan \left \{ \mqty(&&&0&&&\\  &*&&0&&*&\\ &&&0&&&) \right \} \cong Gr(k,n-1) \]

where the rest of the matrix is full-rank. 

Note that $\Pi_{contr,i}$ corresponds to the siteswap $(k-1)^{n-k}k^{k-1}n$ rotated so that the $n$ is in the $i$th position, and $\Pi_{del,i}$ corresponds to the siteswap $(k)^{n-k-1}(k+1)^{k}0$ rotated so that the $0$ is in the $i$th position.

Then the fixed point set under the $S_i$ action is: 

\[Gr(k,n)^{S_i}= \Pi_{contr,i} \sqcup \Pi_{del,i} \cong Gr(k-1,n-1) \sqcup Gr(k,n-1)\]

and every $T$-fixed point is in one or the other.

\end{Prop}

The process called ``contraction" has two steps. First we first intersect with $\Pi_{contr,i}: \Pi_{contr_i(f)} := Contr_i(\Pi_f)= \Pi_f\cap \Pi_{contr,i}$. Here, we are defining $contr_i(f)$ to be the bounded affine permutation of the positroid variety equal to $\Pi_f\cap \Pi_{contr,i}$; this is possible by Proposition \ref{prop:del2def} below, and the details of how to calculate it can be found in Appendix A. This restricts $\Pi_f$ to the subset that can be written in the form of the matrix representative for $\Pi_{contr,i}$ given in the proposition. Second, we do the projection that removes the last row and $i$-th column to bring it to a subset of $Gr(k-1,n-1)$. This intersection with $\Pi_{contr,i}=Contr_i(Gr(k,n))$ turns out to be the largest positroid variety with $f(i)=i+n$ contained in $\Pi_f$ (the proof is below). As long as a given positroid variety $\Pi_f$ contains at least one point whose $k$-plane uses the $i$-th coordinate, the contraction of this positroid variety will be nonempty. 
\\
 
\emph{Remark}: there is sometimes an abuse of terminology where ``contraction" can also refer to only performing the first step, without the projection to $Gr(k-1,n-1)$. We define contraction as performing both steps, and when only the first step is involved, we call it \textbf{projectionless contraction}. Similarly for \textbf{projectionless deletion} below. 
 \\
 
\emph{Remark}: note that contraction (and similarly deletion, which we describe next) is an operation that is performed on the entire Grassmannian $Gr(k,n)$ (in fact $Gr(k,n)=\Pi_{f'}$ with $f'=kk\cdots kk$), but since we will be working with individual positroid varieties throughout this paper, our notation here emphasizes the effect of contraction (or deletion) on a specific positroid variety $\Pi_f$. 
\\
 
Similarly, the process of ``deletion" has two steps: we first intersect with $\Pi_{del,i}: \Pi_{del_i(f)}:= Del_i(\Pi_f)=\Pi_f\cap \Pi_{del,i}$ where $\Pi_{del,i}$.  Then we do the projection that removes the $i$-th column to bring it to a subset of $Gr(k,n-1)$. This intersection with $\Pi_{del,i}=Del_i(Gr(k,n))$ turns out to be the largest positroid variety with $f(i)=i$ contained in $\Pi_f$. 
\\

There is another description of the deletion and contraction of positroid varieties.
For positroid varieties of $Gr(k,n)$, deletion and contraction of the $i$th column of a positroid variety $\Pi_f$ can be described as:

\begin{Prop}\label{prop:del2def}
The projectionless deletion $Del_i(\Pi_f)$ of the $i$-th column of a positroid variety $\Pi_f$ is equal to $\Pi_{f'}^{d(i)}:=$ the largest positroid variety $\Pi_{f'}$ contained in $\Pi_f$ such that $f'(i)=0$. (By largest, we mean that $\Pi_{f'}$ is not properly contained inside another positroid variety $\Pi_{f''}$ properly contained in $\Pi_f$.) 

The projectionless contraction $Contr_i(\Pi_f)$ of the $i$-th column of a positroid variety $\Pi_f$ is equal to $\Pi_{f'}^{c(i)}:=$ the largest positroid variety $\Pi_{f'}$ contained in $\Pi_f$ such that $f'(i)=n$.
\end{Prop}

Recall that the projectionless contraction was defined as the intersection: $Contr_i(\Pi_f)=\Pi_f \cap \Pi_{contr,i}$, and similarly the projectionless deletion is  $Del_i(\Pi_f)=\Pi_f \cap \Pi_{del,i}$.

Before we prove Prop \ref{prop:del2def}, we prove another proposition that we will need in the proof. 

\begin{Prop}\label{prop:delirred}
The intersections $Contr_i(\Pi_f)=\Pi_f \cap \Pi_{contr,i}$ and $Del_i(\Pi_f)=\Pi_f \cap \Pi_{del,i}$ are irreducible.
\end{Prop}

\begin{proof}
Because the positroid varieties form a stratification, intersections of positroid varieties are in general unions of positroid varieties, but we show now that in this case of $Contr_i(\Pi_f)=\Pi_f \cap \Pi_{contr,i}$, this intersection is irreducible (when it's nonempty), which means that $Contr_i(\Pi_f)$ cannot split into multiple positroid varieties - it must be a single positroid variety. The intersection $\Pi_f \cap \Pi_{contr,i}$ is irreducible because we have 
\[\Pi_f \hookleftarrow \Pi_f \backslash \Pi_{del,i} \twoheadrightarrow \Pi_f \cap \Pi_{contr,i}\] 
The latter map is a surjection because we have a right inverse given by taking a plane (point) in $\Pi_f \cap \Pi_{contr,i}$ and taking its direct sum $\oplus \C$, a copy of $\C$ lying along the $i$th coordinate. 
Recall that irreducibility means the coordinate ring has no zerodivisors. Positroid varieties are irreducible, so $\Pi_f$ is irreducible. If we look at $\Pi_f \backslash \Pi_{del,i}$, ripping out a subvariety enlargens the coordinate ring by introducing denominators (localization), not zerodivisors; therefore, $\Pi_f \backslash \Pi_{del,i}$ is irreducible. Finally, if we consider $\Pi_f \backslash \Pi_{del,i} \twoheadrightarrow \Pi_f \cap \Pi_{contr,i}$, we use the fact that the image of an irreducible variety is an irreducible variety to conclude irreducibility of $\Pi_f \cap \Pi_{contr,i}$ (here, the image is everything due to surjectivity, so it is irreducible). This fact is true simply because the map of a variety onto its image corresponds to an inclusion of the coordinate rings in the contravariant direction, and a subring of a ring with no zerodivisors does not have zerodivisors either.  

The analogous statement for $\Pi_f \hookleftarrow \Pi_f \backslash \Pi_{contr,i} \twoheadrightarrow \Pi_f \cap \Pi_{del,i}$ follows from Grassmannian duality and proves the irreducibility of $\Pi_f \cap \Pi_{del,i}$.

\end{proof}

\begin{proof} (of Prop \ref{prop:del2def})

We make use of the fact that $f'(i)=0$ for a positroid variety $\Pi_{f'}$ is equivalent to $\Pi_{f'}$ being contained in $\Pi_{del,i}$; similarly, $f'(i)=n$ for a positroid variety $\Pi_{f'}$ is equivalent to $\Pi_{f'}$ being contained in $\Pi_{contr,i}$. 

$(\supseteq)$ We first show that $Del_i(\Pi_f) \supseteq \Pi_{f'}^{d(i)}$, that is, $\Pi_f \cap \Pi_{del,i} \supseteq \Pi_{f'}^{d(i)}$. This is true because (a) $\Pi_{f'}^{d(i)} \subseteq \Pi_f$ by definition (as the largest positroid variety \textit{contained within} $\Pi_f$) , and similarly $\Pi_{f'}^{d(i)} \subseteq \Pi_{del,i}$ follows from its definition as having $f'(i)=0$, in accordance with the fact that we just stated. The case of contraction is the same.

$(\subseteq)$ Finally we show that $Del_i(\Pi_f) \subseteq \Pi_{f'}^{d(i)}$, that is, $\Pi_f \cap \Pi_{del,i} \subseteq \Pi_{f'}^{d(i)}$. 
For contradiction, assume this is not true; that is, $\Pi_{f'}^{d(i)}$ is properly contained within $\Pi_f \cap \Pi_{del,i}$: $\Pi_{f'}^{d(i)} \varsubsetneqq \Pi_f \cap \Pi_{del,i}$. We will show the contradiction by making use of the irreducibility noted in Prop \ref{prop:delirred} (namely, that while an intersection of positroid varieties like $\Pi_f \cap \Pi_{del,i}$ can in general split into a union of positroid varieties: $ \Pi_{f_1} \cap \Pi_{f_2} = \Pi_{f_3} \cup \Pi_{f_4} \cup \cdots \cup \Pi_{f_j}$, in this special case of $\Pi_f \cap \Pi_{del,i}$, the irreducibility implies that it is a single positroid variety: $\Pi_f \cap \Pi_{del,i} = \Pi_g$). In general, $\Pi_g \varsubsetneqq \Pi_f$, so we would have $\Pi_{f'}^{d(i)} \varsubsetneqq \Pi_g \varsubsetneqq \Pi_f$, and this contradicts the defintion of $\Pi_{f'}^{d(i)}$ as the largest positroid variety contained in $\Pi_f$ such that $f'(i)=0$ (since $\Pi_g = \Pi_f \cap \Pi_{del,i}$ also satisfies $f'(i)=0$). Of course, in the special case when $\Pi_f \subseteq \Pi_{del,i}$, then $\Pi_f \cap \Pi_{del,i} = \Pi_f$, and the largest positroid variety $\Pi_{f'}$ contained within $\Pi_f$ satsifying $f'(i)=0$ is just $\Pi_f$ itself, so $\Pi_{f'}^{d(i)} = \Pi_f = Del_i(\Pi_f)$ and so everything holds automatically.

The proof for the contraction case is analogous.
\end{proof}

\section{Affine Pipe Dreams}\label{section:affinepipe}

\begin{Def}[Pipe Dreams for Matrix Schubert Varieties]
A pipe dream of size $n$ is a diagram in an $n$-by-$n$ square where each box can be filled by one of two types of tiles, elbows \- and crosses \+, such that all crosses occur above the antidiagonal. Thus, we can think of the grid as a set of pipes that begin on the north and east edges and end on the west and south edges (with the south to east tiles being all elbows).

We called the pipe dream \textbf{reduced} if no two of the pipes cross twice. We will focus only on the reduced case in the following, so the adjective ``reduced" should always be assumed. 
\end{Def}

\begin{Exa}
The following figure shows an example for all pipe dreams corresponding to the permutation $\pi=2143$ in $S_4$. The pipes running from the south edge to the east edge are uninteresting so are not shown. 
\\

\end{Exa}

\begin{figure}[htbp]\includegraphics[scale=0.5,clip=true]{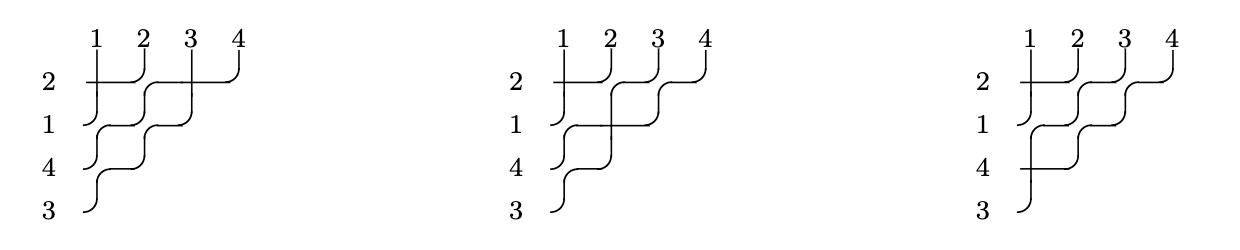}\end{figure}

\begin{Def}\label{def:moves}
A \textbf{cross-elbow move} (in the following: ``\textbf{move}," for short) in a pipe dream is an interchange of a cross tile and an elbow tile that preserves the connectivity of the pipes - in other words, that leaves the permutation invariant. We think of the cross as ``moving" to the position where the elbow it is replacing was located. If we perform all possible moves in a pipe dream, and keep repeating, taking all the new pipe dreams obtained after performing moves and doing moves on these new ones, then we will obtain the set of all possible pipe dreams for that particular permutation (see \cite{BB93}). If the pipe dream has no possible moves, so is the unique pipe dream for a particular permutation, then we call the pipe dream \textbf{rigid}. The notion of moves carries over without change to the case of bounded affine permutations which we define next (the chute and ladder moves of \cite{BB93} are a subset of our moves), and will be used extensively below. 
\end{Def}

\begin{Exa}
In the previous figure, for the pipe dream on the very left, we see that the cross in the northwest corner has no moves, but the other cross (to the right) has 2 different places it could potentially move, which does not change the connectivity/permutation. Therefore, this gives us the 3 different pipe dreams. 
\end{Exa}

There is a set of \textbf{affine pipe dreams} associated to a pair $(f, \lambda)$ of a bounded affine permutation $f$ and a set of $k$ columns $\lambda \in {[n] \choose k}$, which diagrammatically display the intersection $\Pi_f \cap U_\lambda$. The shape or outer boundary of the affine pipe dream depends only on $\lambda$; it is constructed by picking a \textbf{distinguished path} corresponding to $\lambda$ and using it to cut a matrix into two pieces, which are then shifted northeast and southwest infinitely. The details are in Chapter 4 of \cite{Sni}, as well as in this section and the Appendix.

\begin{Exa}
In this example, the vertical lines in the picture on the left occur right next to the $1$'s where the columns in $\lambda$ are: these are columns 1, 2 and 4. On the right, after ``collapsing" these columns, they are now depicted by the vertical segments in the distinguished path (the staircase line) running NW to SE. 

\begin{figure}[htbp]\includegraphics[scale=0.7,clip=true]{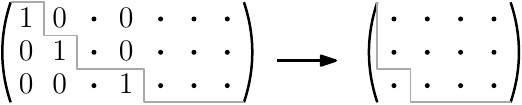}\end{figure}

\end{Exa}

The affine pipe dream's shape is an infinite strip obtained by taking this staircase line and placing an exact copy directly above it and another copy to the right. This continues forever in the NW and SE directions as depicted in the figure below (note that later in this section we will explain the numbers within the pipe dream shape). Affine pipe dreams are diagrams drawn by filling in the squares inside the shape with cross and elbow tiles. 

\begin{figure}[htbp]
\includegraphics[scale=0.7,clip=true]{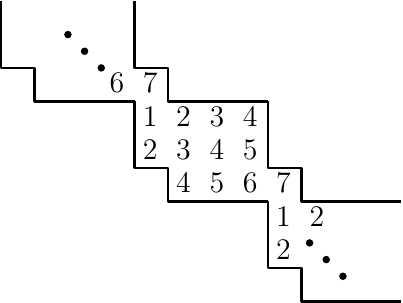}
\end{figure}

The main theorem proven in \cite{Sni} is that the restriction of a positroid variety to a point is stratified isomorphic to affine Kazhdan-Lusztig varieties. In the next section, we will use this to show that this implies that the equivariant cohomology class of a positroid variety restricted to a point is therefore given by a sum over these affine pipe dreams (for the shape of the point being restricted to). We will then show that by restricting our torus action to a circle acting by dilation, we obtain a corollary that the multiplicity of that point on the positroid variety is equal to the number of pipe dreams for that positroid variety in that shape. The details are below. In any case, because \cite{Sni} proves this correspondence using these pipe dreams shapes, and since this correspondence is all we need from these shapes to prove the theorem on the relationship between smoothness and pipe dreams in the next section, we can use them for the results we prove below. However, it is not difficult to prove that any periodic pipe dream shape, which when filled in entirely with crosses produces the permutation of a point, must be of the form above; the proof is included in a companion thesis \cite{MyThesis}. 

The numbers within the pipe dream shape come from thinking of the pipe dream as a rotated wiring diagram, and subsets of these numbers form words in a Coxeter group; for example, the numbers along the northeast diagonal coming out from the ``$1$" spot along the lower boundary are also all labeled ``$1$" since any cross placed along this diagonal should produce an $s_1$ in the word - that is, it should transpose the first and second elements, or equivalently in the picture, switch the pipes emerging from the first and second positions. 
\\

Now we show that for these affine pipe dreams, it is still the case that if a particular affine pipe dream is rigid (it has no possible moves), then it is the only one that produces its corresponding bounded affine permutation.

\begin{Prop}\label{prop:BAPmoves}
For any affine pipe dream $\Omega$ with bounded affine permutation $g$, every other affine pipe dream $\Omega'$ that also has bounded affine permutation $g$ can be reached via $\Omega$ by performing a sequence of moves on $\Omega$.
\end{Prop}

Before we give the proof, we state the following definition of a word, taken from \cite{KM03}, and a couple of facts to be used in the proof.

\begin{Def}[\cite{KM03}]
Let $\Pi$ denote the set of permutations on $n$ elements, and let $\Sigma$ denote the set of transpositions $s_1,...,s_n$ where $s_i$ switches $i$ and $i+1$. A word of size $m$ is an ordered sequence  $Q=(\sigma_1,...,\sigma_m)$ of elements of $\Sigma$. An ordered subsequence $P$ of $Q$ is called a \textbf{subword} of $Q$. 

 $P$ \textbf{represents} $\pi\in \Pi$ if the ordered product of the simple reflections in $P$ is a reduced decomposition for $\pi$.

$P$\textbf{contains} $\pi\in\Pi$ if some subsequence of $P$ represents $\pi$.

The subword complex $\Delta(Q,\pi)$ is the set of subwords $Q\backslash P$ whose complements $P$ contain $\pi$.
\end{Def}

In Appendix B, we do an example to illustrate how a subword $P_{f,\lambda}$ for the siteswap $f$ at the point $\lambda$ is obtained from the word $Q_\lambda$ of the affine patch around the point $\lambda$. Here, we just state the following facts. 

\begin{Prop}[Corollary 1 in \cite{Sni}]\label{prop:subwdcomp} 
Define the \textbf{affine pipe dream complex} $\Delta(\lambda,\pi)$ as is the simplicial complex with vertices labeled by entries $(i, j)$ in the periodic strip and faces labeled by the elbow sets in the affine pipe dreams for $\pi$ with shape defined by $\lambda$. It is isomorphic to the subword complex $\Delta(Q_\lambda,f)$.
\end{Prop}

\begin{Thm}[KM04, Theorem 3.7]\label{thm:ballsphere} The topological realization of any subword complex $\Delta(Q,\pi)$ homeomorphic to a ball or sphere; in particular, every ridge (codimension 1 facet) is contained in one or two facets. The dual graph (whose vertices are the facets of $\Delta(Q,\pi)$ and whose edges are the ridges of $\Delta(Q,\pi)$) is connected.
\end{Thm}

\begin{Def}[Definition \# 2 of pipe dreams moves]
We defined pipe dream moves above in terms of exchanges of tiles on a particular pipe dream. In the context of subword complexes, moves correspond to going to a ridge and then going to another codimension-1 facet. This follows from the fact the notion give above of a move in terms of exchanging a cross and a near-miss can actually be thought of as adding a cross in the new location before removing the original one, and the ridges in the complex are where we have nonreduced pipe dreams.
\end{Def}

\begin{proof}[Proof of Prop \ref{prop:BAPmoves}]
By Proposition \ref{prop:subwdcomp}, there exists a subword complex $\Delta(Q_\lambda,f)$ for which the set of all pipe dreams for a siteswap $f$ on a patch $\lambda$ are the codimension-1 facets of this complex. By Theorem \ref{thm:ballsphere}, since subword complexes are homeomorphic to balls or spheres, they are connected in codimension-1. Therefore, any two pipe dreams for the same siteswap on the same patch must be connected by a sequence of moves. 
\end{proof}

We conclude this section with a picture of what $\Pi_{del,i}$ looks like on an open patch. Recall that this is the positroid variety which has siteswap $(k+1)^k0k^{n-k-1}$ rotated so the $0$ is in the $i$ slot. It is the largest positroid variety whose siteswap has a $0$ in the $i$-th slot. It will be easy to see that the pipe dream must look like the following: 

\begin{figure}[htbp]\includegraphics[scale=0.45,clip=true]{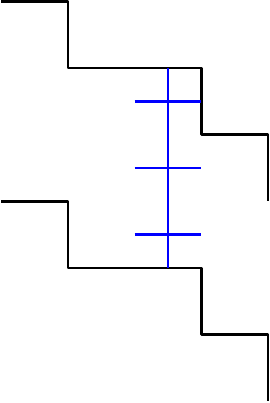}\end{figure}

Here is the proof. The siteswap is $(k+1)^k0k^{n-k-1}$. In the pipe dream above, the crosses are placed along the $i$-th slot, so that the only difference from the pipe dream with all elbows (producing the siteswap $kk\cdots kk$), is that the $k$ crosses along that vertical line cause the $k$ slots immediately prior to $i$ to move one slot ahead of what they would be in the all-elbow pipe dream. This is precisely the $(k+1)^k$ part immediately prior to the $0$ in the siteswap. Thus, this pipe dream produces the siteswap we wanted, and in fact it is the unique such pipe dream since there are no moves in the pipe dream. 

Another way to prove this is that, by the proof above, not only must the column to be deleted lie along a horizontal segment, but the entire column must be filled with crosses so that it ends in the minimal spot (in this case: itself, given by the siteswap notation of $0$ steps forward). Any additional crosses lowers the dimension of the variety. $\Pi_{del,i}$ was defined as the maximal positroid variety with a $0$ in the $i$-th slot, so it must be depicted by this pipe dream. 
\\

The case of  $\Pi_{contr,i}$ (which has siteswap $(k-1)^{n-k}k^{k-1}n$ rotated so the $n$ is in the $i$ slot) is analogous, with the crosses proceeding from the vertical line segment labeled $i$ in a horizontal line from the left to the right. 
\\

The next question to be asked is then what the deletion or contraction of a positroid variety looks like. We know what $\Pi_{del,i}$ and $\Pi_{contr,i}$ look like, from what we have just done. Previously, we defined the deletion or contraction of a positroid variety $\Pi_f$ as the intersection of $\Pi_f$ with $\Pi_{del,i}$ or $\Pi_{contr,i}$. So the question now is what this intersection looks like on an affine patch. We have the following proposition:

\begin{Prop}\label{prop:delpipedream}

The pipe dreams of the $i$-th deletion of a positroid variety $\Pi_f$ (that is, $del_i(\Pi_f)=\Pi_f \cap \Pi_{del,i}$) on the open patch $U_\lambda$ are given by taking the pipe dreams for $\Pi_f$ on $U_\lambda$ with the maximal number of crosses along the $i$th column, and then filling in this $i$th column with crosses.
\end{Prop}

\begin{Exa}
Here is a quick pictorial representation of $Del_4(\Pi_f)$ (assuming the blank parts of the pipe dream are all elbows): 

 \begin{figure}[htbp] \centering
	\includegraphics[scale=0.5,clip=true]{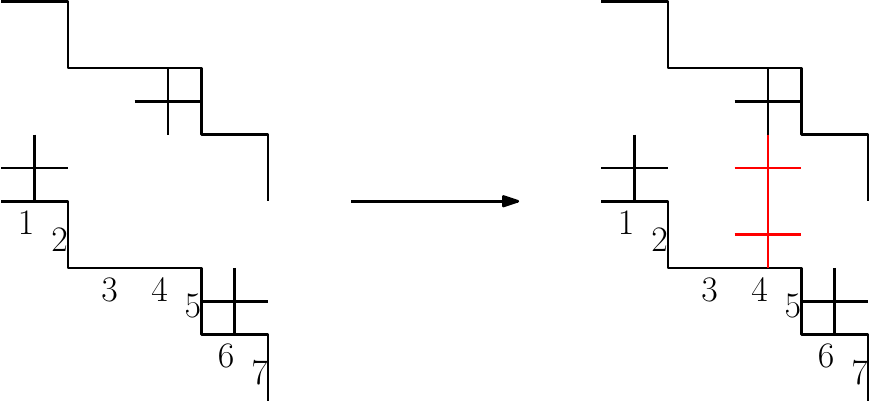}
\end{figure}

That is, we take the pipe dream(s) with the largest number of crosses along that column, and fill in the rest of this column with crosses (denoted by the red crosses in this example here).
\end{Exa}

We can read out the new (deleted) affine permutation by following the pipes along this new pipe dream; however, note that the affine permutation $del_i(f)$ is not dependent on which patch $U_\lambda$ we draw the pipe dream on. Thus, if the goal was simply to compute $del_i(f)$, rather than drawing the pipe dream (on a particular patch $U_\lambda$), we could do this more efficiently using an algorithm that we present in Appendix A.  

The case of contraction is analogous with the crosses filled in horizontally (in a row rather than a column of crosses).

\begin{proof}[Proof of Prop \ref{prop:delpipedream}]

We take $Del_i(\Pi_f)=\Pi_f \cap \Pi_{del,i}$ and intersect everything with the vector space $U_\lambda$:

\[Del_i(\Pi_f)\cap U_\lambda=(\Pi_f \cap U_\lambda) \cap (\Pi_{del,i} \cap U_\lambda).\] 

In other words, we get the components of $(\Pi_f \cap U_\lambda) \cap (\Pi_{del,i} \cap U_\lambda)$ by taking the components of the pipe dreams for $\Pi_f$ and the components of the pipe dreams for $\Pi_{del,i}$ and intersecting them and looking at the maximal element. Indeed, as we showed in the proof of Prop \ref{prop:del2def}, the intersection $Del_i(\Pi_f)=\Pi_f \cap \Pi_{del,i}$ is irreducible so is a single maximal positroid variety. We then intersect this positroid variety with $U_\lambda$, which is just looking at the pipe dreams for $Del_i(\Pi_f)$ on the shape defined by $\lambda$. 

Furthermore, since we know that the maximal intersection will have the fewest number of crosses (since each cross represents a variable that is set to zero - thus an extra condition that lowers the dimension of the variety), and since we know that deletion involves filling in the column entirely with crosses, we must choose to intersect with the pipe dream(s) for $(\Pi_f \cap U_\lambda)$ that already had the greatest number of crosses along this column so that filling in this column involves adding the minimal number of additional crosses.

\end{proof}

\section{Positroid Varieties and Smoothness}

We have already seen how to study positroid varieties on an open cover using affine pipe dreams. We have also seen how easy it is perform deletion and contraction on positroid varieties by using these affine pipe dreams. The main goal of this section is to prove Theorem \ref{thm:smoothone}, which will provide a means of studying smoothness or singularity of positroid varieties on this open cover using pipe dreams.

\begin{Thm}
One can check whether a positroid variety $\Pi_f$ is singular (or smooth) just by testing singularity at the fixed points $\lambda$.
\end{Thm}

\begin{proof}

Fact (1): Borel's Fixed Point Theorem: If a solvable group (such as a torus) acts on some nonempty proper (such as projective) variety, then the fixed point set is also nonempty.
\\

Fact (2): The torus $T$ acts on the singular locus.

Proof: Whenever a group $G$ acts on a variety, the singular locus will be $G$-invariant (i.e. a group of symmetries will leave the singular locus alone). 
\\

Fact (3): Positroid varieties are proper. 
Positroid varieties are closed in Grassmannians which are closed in projective space (in the Plücker embedding).
\\

Fact (4): The singular locus of any variety is closed. 
\\

Suppose $\Pi_f$ has singular point(s), i.e. the singular locus $(\Pi_f)_{sing}$ is nonempty. Then we know that $(\Pi_f)_{sing}$ is closed in $\Pi_f$ by (4). Since $\Pi_f$ is proper by (3) and $(\Pi_f)_{sing}$ is closed in $\Pi_f$, we conclude that $(\Pi_f)_{sing}$ is also proper. By (2), the torus acts on the singular locus, so we can apply (1) Borel's Fixed point Theorem to it, to conclude that the singular locus includes at least one fixed point (under the torus action).  

Denoting the fixed points of the torus action with superscript $T$, we summarize:
\[ (\Pi_f)_{sing} \neq \emptyset \Leftrightarrow (\Pi_f)_{sing}^T \neq \emptyset\]
\end{proof}

In what follows, we will check these by checking on the affine open cover of the $U_\lambda$ described earlier. Clearly, this will include all the torus-fixed points $ \{\lambda\}$.

\subsection{Smoothness Theorem}\label{subsect:smooththm}

\begin{Thm}\label{thm:smoothone}
A positroid variety $\Pi_f$ is smooth at the point $\lambda$ (it meets $\lambda$ and is smooth there) if and only if there is a single affine pipe dream representative \cite{Sni} of the affine permutation $f$ on the open set $U_\lambda$. Specifically $\Pi_f$ meets the point $\lambda$ if and only if there exists (at least) one affine pipe dream, and it is smooth if and only if that affine pipe dream is unique. Thus, we can take $\Pi_f$ and look at pipe dreams for all choices of $\lambda$ to test for smoothness. 
\end{Thm}

We will spend most of the rest of this section proving this theorem. We begin by developing some of the background facts.

\begin{Def} 
The \textbf{multiplicity} $Mult(x\in X)$ is the degree of the tangent cone.
\end{Def}

\begin{Thm}
[\cite{Hartshorne} Exercise 5.3a] $Mult(x\in X)=1$ if and only if $X$ is regular (or smooth) at $x$. 
\end{Thm}

Thus, proving that ``smoothness means one pipe dream" has become ``multiplicity 1 means one pipe dream"; the next theorem shows that alternatively we have ``coefficient 1 in equivariant cohomology means one pipe dream." For an introduction to equivariant cohomology, especially in this context, see for example sections 1.2, 2.1 of \cite{KnTao01}. Note that the idea of using equivariant cohomological techniques to compute multiplicities of points was used in for example \cite{GraKre}.

\begin{Thm}\label{thm:Rossmann}\cite{Rossmann}
Let $V$ be a vector space, and $X \subseteq V$ be a cone, i.e. a subvariety that is invariant under the dilation action of $\C^\times$ on $V$ (that is, $\C^\times$ acts with all weights $1$ on $V$). Denote the equivariant cohomology $H_{\C^\times}^*(V)$ by $\Z[h]$. Then the multiplicity of $X$ at $0\in V$ is the coefficient of $h^{\codim X}$: 
\[[X]_{\C^\times} = (\deg X)h^{\codim X} \in H^*_{\C^\times}(V)\]
\end{Thm}

Remark: We use the notation $[X]_{\C^\times}$ to disambiguate from $[X]_T \in H^*_T(V)$, which we will see later. 

In our case, this means that we want to show that $[\Pi_f \cap U_\lambda] = 1 \cdot h^{codim X} \in H_{\C^\times}^*(U_\lambda)$ is equivalent to $\Pi_f$ having a unique pipe dream on $U_\lambda$. More generally, we will show that the multiplicity counts the number of affine pipe dreams. 

The proof goes by showing the following sequence of equalities:

\begin{equation*}\label{equivcoheq} 
(*) [\Pi_f \cap U_\lambda \subseteq U_\lambda] \overset{(1)}{=} [X_{v(f)} \cap X_\circ^{v(\lambda)} \subseteq X_\circ^{v(\lambda)}] \overset{(2)}{=} [X_{v(f)}]|_{v(\lambda)} \overset{(3)}{=} \sum_R \prod_{r\in R}  \hat{\beta_r} \in H_T^*(U_\lambda) \cong \Z[y_1,...,y_n]
\end{equation*}

The following subsections will individually go through equalities (1), (2), and (3) in this equation labeled $(*)$, as well as a last part (4) which maps from the $T$-equivariant cohomology to a circle $S$ within the torus $T$.

Note on notation: $f$ is a siteswap and $v(f)$ its corresponding bounded affine permutation, which indexes a $T$-fixed point on the affine flag variety $\widehat{GL_n}/\widehat{B}$ (note: $\widehat{GL_n}$ denotes $n\times n$ invertible matrices where the entries are Laurent series in one variable, and this is infinite-dimensional). Then $X_\circ^{v(\lambda)}$ is the $B$-orbit through the point $v(\lambda)B/B$; $X_\circ^{v(\lambda)} := B v(\lambda) B/B$. And  $X_{v(f)}$ is the $B_-$-orbit closure through $v(f)$; $X_{v(f)} := \overline{B_- v(f) B}/B$.

\subsubsection{Equation $(*)$ Part (1)}

Equality (1) follows from the work of \cite{Sni} Section 4.2, which says we have a T-equivariant commutative diagram:

$$
 \begin{tikzcd}
U_\lambda \ar[r,leftarrow,"\cong"]  & X_\circ^{v(\lambda)} \\
\Pi_f \cap U_\lambda \ar[r,leftarrow,"\cong"] \ar[u,hookrightarrow] & X_{v(f)} \cap X_\circ^{v(\lambda)} \arrow[u,hookrightarrow] 
\end{tikzcd}
$$

On the level of T-equivariant cohomology, this gives rise to:
\\

$$
\begin{tabular}{ccc}
$g^*: H_T^*(U_\lambda)$ & $\ra{\cong}$ & $H_T^*(X_\circ^{v(\lambda)})$ \\
$\rotatebox[origin=c]{90}{$\in$}$ & & $\rotatebox[origin=c]{90}{$\in$}$ \\
$[\Pi_f \cap U_\lambda ]$ & $\mapsto$ & $[X_{v(f)} \cap X_\circ^{v(\lambda)} ]$ 
\end{tabular}
$$

Since $U_\lambda$ and $X_\circ^{v(\lambda)}$ are contractible, we have canonical isomorphisms: $H_T^*(U_\lambda) \cong H_T^*(pt) \cong H_T^*(X_\circ^{v(\lambda)})$. 
$$\begin{tikzcd} & H_T^*(pt) \ar[dl,leftrightarrow,"\cong"'] \ar[dr, leftrightarrow,"\cong"] & \\ H_T^*(U_\lambda)  \ar[rr,"g^*"] & & H_T^*(X_\circ^{v(\lambda)}) \end{tikzcd}$$ 
Using these isomorphisms to identify each with $H_T^*(pt)$, the polynomials $[\Pi_f \cap U_\lambda]$ and $[X_{v(f)} \cap X_\circ^{v(\lambda)} ]$ are equal. 

This completes the justification of equality (1). Next we prove equality (2): $[X_{v(f)} \cap X_\circ^{v(\lambda)} \subseteq X_\circ^{v(\lambda)}] = [X_{v(f)}]|_{v(\lambda)}$.

\subsubsection{Equation $(*)$ Part (2)}

The right side of this equality is the restriction of the class $[X_{v(f)}]$ to the point $v(\lambda)B/B$ (in the following, we will denote this as simply $v$ or $v(\lambda)$). Naively, we might consider this via the diagram: 

$$\begin{tikzcd}
 & X_{v(f)} \ar[d, hookrightarrow] \\
v  \ar[r, hookrightarrow, "\iota"]  & G/B 
\end{tikzcd}
$$

We factor this inclusion as follows:

$$\begin{tikzcd}
& X_{v(f)} \cap X_\circ^{v(\lambda)} \ar[d, hookrightarrow]  \ar[r] & X_{v(f)} \ar[d, hookrightarrow] \\
v  \ar[r, hookrightarrow] & X_\circ^{v(\lambda)} \ar[r, hookrightarrow]  & G/B 
\end{tikzcd}
$$

The bottom row induces maps in cohomology:
\[H^*_T(v)  \leftarrow  H^*_T(X_\circ^{v(\lambda)})  \leftarrow H^*_T(G/B) \]
The goal is to show that under this induced map:
\[ [X_{v(f)}] |_{v} \mapsfrom  [X_{v(f)} \cap X_\circ^{v(\lambda)}]  \mapsfrom [X_{v(f)}] \]

We know that the composition of the two maps will take $[X_{v(f)}]$ to $[X_{v(f)}] |_{v}$ because cohomology is functorial so the pullback on cohomology of the composite map $v \to X_\circ^{v(\lambda)} \to G/B$ is the composite of the pullback; in other words, once we know where $v$ maps to, given by the map $v\hookrightarrow G/B$, this determines the map $H^*_T(G/B) \to H^*_T(v)$ regardless of how we might factor the map in the middle, and this defines $[X_{v(f)}] |_{v}$ as the image of this map. 

The harder thing is we need to show that this map hits $[X_{v(f)} \cap X_\circ^{v(\lambda)}]$ in the factorization through the middle space $X_\circ^{v(\lambda)}$. This will follow from transversality of $X_{v(f)}$ and $X^{v(\lambda)}_\circ$ \cite{Kleiman}. We can apply Kleiman transversality because $X_\circ^{v(\lambda)}$ is a $B$-orbit, and and $X_{v(f)}$ is $B_-$ invariant, so then we conclude that under the map $g: X_\circ^v \hookrightarrow G/B$ going to $g^*: H_T^*(G/B) \to H_T^*(X_\circ^v)$, $g^*([X_{v(f)} ]) = [X_{v(f)} \cap X_\circ^{v(\lambda)} ]$.

Finally, to prove equality (2), we need to show that  $[X_{v(f)} \cap X_\circ^{v(\lambda)}]$ doesn't just map to $[X_{v(f)}] |_{v}$, but that they are equal: we could see this already from the discussion preceding the triangle commutative diagram above, since $X_\circ^{v(\lambda)}$ is equivariantly contractible to $v$, so $H_T^*(v) \cong H_T^*(X_\circ^{v(\lambda)})$. 

This completes the proof of equality (2).

\subsubsection{Equation $(*)$ Part (3)}
 
Equality (3) is the AJS/Billey formula \cite{AJS/Billey}, which gives the restriction of a Schubert class to a fixed point. We state it here:

\begin{Def}
Let $Q$ be a word in the simple roots of a Kac-Moody group (see \cite{Kumar} Prop 11.1.11 for the Kac-Moody case). We define $\beta_i$ as follows:
\[\beta_i := (\prod_{j=1}^{i-1} r_{Q(j)})\cdot (x_{Q(i)} - x_{Q(i)+1}) =  r_{Q(1)}\cdots r_{Q(i-1)} \cdot (x_{Q(i)} - x_{Q(i)+1}), i\in [m] \]
\end{Def}
Let $\prod Q$ be a word in the letters $Q(i)$ so $\prod Q$ is an element of a Coxeter group (in our case, this will be a product of transpositions representing a bounded affine permutation), and $[X_w]|_{\prod Q}$ denote the restriction of the Schubert class $[X_w]$ to the point in the (affine) flag variety represented by $\prod Q$. Then the \textbf{AJS/Billey formula} says:

\[[X_w]|_{\Pi Q}= \sum_{\substack{R\subseteq Q \\ R\text{ reduced}}} \prod_{r\in R}  \hat{\beta_r}  \]
\\

Our work so far has proven:
\[ [\Pi_f \cap U_\lambda]_T = \sum_{\substack{R\subseteq Q \\ R\text{ reduced}}} \prod_{r\in R}  \hat{\beta_r} \]
We state a proposition for what the right hand side looks like.

\begin{Prop}\label{prop:AJSrowcolform}
\[ [\Pi_f \cap U_\lambda]_T = \sum_{\substack{R\subseteq v(\lambda) \\ \prod R=v(f) \\ R \text{ reduced}}} \prod_{r\in R}  [(y_i)_r - (y_j)_r] \]

The AJS/Billey formula produces a sum of products of $(y_i-y_j)$: the terms multiplied together to form a product are indexed by $r\in R$, and each $r$ corresponds to a cross at $(j,i)$ in a pipe dream representative for the given positroid variety on the given patch. Specifically, the $i$ corresponds to the column and the $j$ to the row in the pipe dream. The labelling is given by the numbering of the distinguished path explained in Appendix B; thus, the rows $j$ will come from the vertical boundary line segments, which correspond to $\lambda$, where the columns $i$ will be from the horizontal boundary line segments found in $[n]\backslash \lambda$. 
\end{Prop}

\begin{proof}

First, we note that we can think about the letters or transpositions in the affine pipe dream shape as SW to NE diagonals where all the letters are the same. See the following figure:

\begin{figure}[htbp] \centering
	\includegraphics[scale=0.7,clip=true]{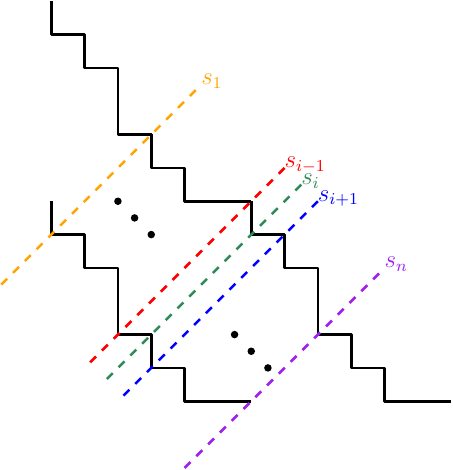}
\end{figure}

\newpage

The important thing to note is that each $s_i$ commutes with all other $s_j$ except those in its neighboring adjacent diagonals $s_{i-1}$ and $s_{i+1}$. This implies that the invariant information of a word can be depicted in a heap (see e.g. \cite{Stem} for details), since a heap diagram only keeps track of the information about the order in which adjacent $s_i$'s are read. Any two words with the same heap will give the same permutation (c.f. for example \cite{BJN}). 

The natural heap our affine pipe dream shapes provide can be obtained by rotating the pipe dream 45 degrees counterclockwise (c.f. the ``bottom pipe dream" of Appendix B): 

\begin{figure}[htbp] \centering
	\includegraphics[scale=0.6,clip=true]{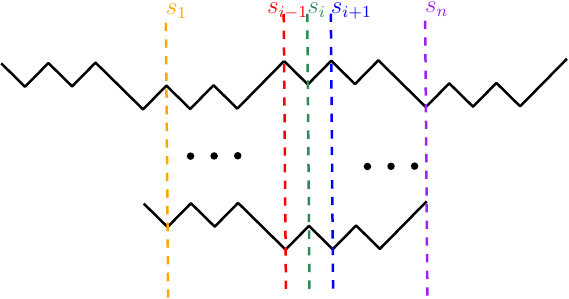}
\end{figure}

There is a question of how to read the letters to form a word. This heap corresponds to a southwest (denoted SW for short) rule for reading, where we always only choose a (letter in a) square $s_i$ if all squares southwest have already been chosen. In other words, regardless of how we choose to read the letters within the affine pipe dream, any choices of order that satisfy the SW rule will all be equivalent because they have the same heap (they will be related by commuting moves, since all non-commuting moves will be the same, being determined by the SW rule).
\\

Thus, in order to prove the proposition, we just need to show 2 things: (1) that the way of reading the word in \cite{Sni}, which is the way that we read it in this paper, satisfies the SW rule, and (2) we can choose an order satisfying the SW rules that makes it apparent that when we calculate the root $\beta_i=\prod_{j=1}^{i-1} r_{Q(j)} (x_{Q(i)} - x_{Q(i)+1})$, it equals $x_{row}-x_{col}$. 

For (1), it will be easiest to follow the argument if an example is presented. We give an example in Gr(4,10): (After this proof is completed, we continue this example in more detail.) 

\newpage

\begin{figure}[htbp] \centering 	\includegraphics[scale=0.5,clip=true]{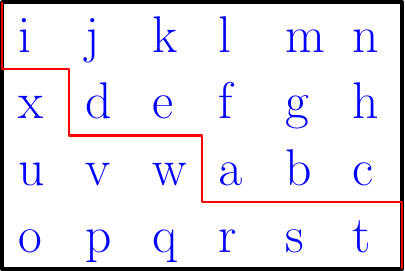} \end{figure}

Thus, the word would be: $Q(a)Q(b)Q(c)Q(d)Q(e)Q(f)Q(g)Q(h)Q(i)Q(j)Q(k)Q(l)Q(m)$ $Q(n)Q(o)Q(p)Q(q)Q(r)Q(s)Q(t)Q(u)Q(v)Q(w)Q(x)$.

The way of reading the word in \cite{Sni}, which we use in this paper, satisfies the SW rule because we start in the left corner or the lowest row, which has a boundary to its south. We always read left to right in a row, with all letters to the south already chosen since we read starting from lower rows. This proves that we satisfy the SW rule.  

(2) Thus, our work in part (1) shows that $\beta_i=\prod_{j=1}^{i-1} r_{Q(j)} (x_{Q(i)} - x_{Q(i)+1})$ equals any product $\prod r_{Q(l)} (x_{Q(i)} - x_{Q(i)+1})$ where $\prod r_{Q(l)}$ satisfies the SW rule. 

We make one observation, which is the affine pipe dream diagram has width $k$ for any row, and height $n-k$ for any column, so this is in particular the maximum width and height if we go west and south from any square in the pipe dream:

\begin{figure}[htbp] \centering
	\includegraphics[scale=0.6,clip=true]{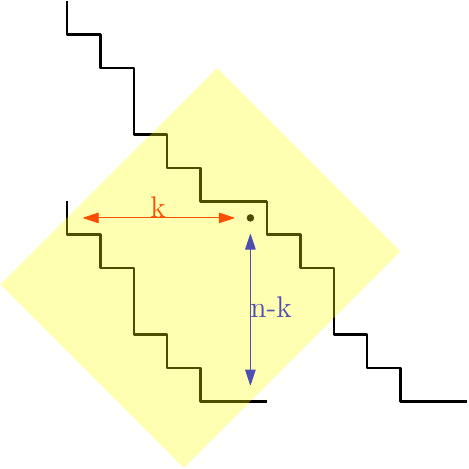}
\end{figure}

This means that if we go left from $s_i$, the leftmost square will be a maximum of $k-1$ steps left, and if we go south, the southernmost square will be a maximum $n-k-1$ steps down, so the entire region between where it hits the left boundary and south boundary (highlighted in yellow in the previous figure) will have no overlaps, and will contain $n-1$ diagonals $s_{j\mod n}, s_{j+1\mod n},...,s_{j-2\mod n}$. 
 
Now, for any square $x$, we consider just the word for which $x$ is the very northeastern corner, and otherwise, the word contains the squares within the area to its southwest, e.g.:

\begin{figure}[htbp] \centering
	\includegraphics[scale=0.6,clip=true]{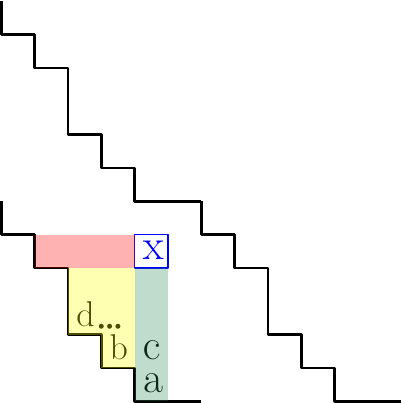}
\end{figure}

We read the word start from the lowest row, row by row upwards, and left to right within each row, as depicted a,b,c,d,... in the example figure above. This satisfies the SW rule. This implies that $\beta_i=\prod_{j=1}^{i-1} r_{Q(j)} (x_{Q(i)} - x_{Q(i)+1})$ has the form: 
\begin{align*}
\beta_i &= \cdots \teal{r_{Q(i)+m}} \cdots \teal{r_{Q(i)+2}} \cdots \teal{r_{Q(i)+1}} \rd{r_{Q(i)-l} \cdots r_{Q(i)-2}  r_{Q(i)-1} }(x_{Q(i)} - x_{Q(i)+1}) \\
&= \cdots \teal{r_{Q(i)+m}} \cdots \teal{r_{Q(i)+2}} \cdots \teal{r_{Q(i)+1}} \rd{r_{Q(i)-l} \cdots r_{Q(i)-2} }(x_{Q(i)-1} - x_{Q(i)+1}) \\
&= \cdots \teal{r_{Q(i)+m}} \cdots \teal{r_{Q(i)+2}} \cdots \teal{r_{Q(i)+1}} (x_{row} - x_{Q(i)+1}) \\
&= \cdots \teal{r_{Q(i)+m}} \cdots \teal{r_{Q(i)+2}} \cdots  (x_{row} - x_{Q(i)+2}) \\
&= x_{row} - x_{col} 
\end{align*}

Note that the red reflections correspond to the row highlighted red in the figure above, and the green reflections to the column highlighted green in the figure above. Given the way the word is read, the red reflections will act first, bringing $x_{Q(i)}$ all the way to $x_{row}$, where $row$ is the label on the vertical line segment to the very left of $x$. For the remaining reflections, only the green ones will cause any change, and they will act by bringing $x_{Q(i)+1}$ all the way to $x_{col}$, where $col$ is the label on the horizontal line segment to the very south of $x$. By the observation above, each diagonal of $s_i$'s only occurs at most once in this region, so there will be no interference between the letters in the region highlighted red and the letters in the region highlighted green.
 
\end{proof}

\begin{Exa}
We continue the example in Gr(4,10) that we introduced in the proof above. Below, we will calculate $\beta_i=\prod_{j=1}^{i-1} r_{Q(j)} (x_{Q(i)} - x_{Q(i)+1})$, showing that it equals $x_{row}-x_{col}$. 
For this example, let us choose $Q(i)$ to be the letter circled in green, where we have highlighted in orange $\prod_{j=1}^{i-1} r_{Q(j)} = r_{Q(1)} r_{Q(2)} r_{Q(3)} \cdots r_{Q(i-1)}$: 

\begin{figure}[htbp] \centering
	\includegraphics[scale=0.5,clip=true]{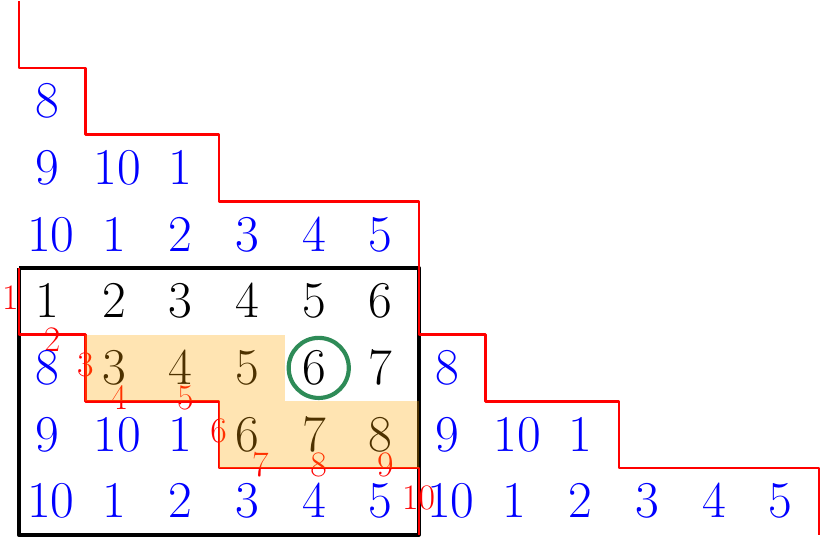}
\end{figure}

We obtain:
\begin{align*}
\beta_i &= r_{Q(1)} r_{Q(2)} r_{Q(3)} \cdots r_{Q(i-1)} (x_{Q(i)} - x_{Q(i)+1}) \\
&= r_{Q(1)} r_{Q(2)} r_{Q(3)} r_{Q(4)} r_{Q(5)} r_{Q(6)} (x_{Q(7)} - x_{Q(7)+1}) \\
&= r_6 r_7 r_8 r_3 r_4 r_5 (x_6 - x_7) \\
&= r_6 r_7 r_8 r_3 r_4  (x_5 - x_7) \\
&= r_6 r_7 r_8 r_3  (x_4 - x_7) \\
&= r_6 r_7 r_8 (x_3 - x_7) \\
&= r_6 r_7  (x_3 - x_7) \\
&= r_6 (x_3 - x_8) \\
&= x_3 - x_8
\end{align*}

and this is precisely $x_{row}-x_{col}$ if we draw lines west and south in the figure, as we wanted to show.

\end{Exa}

We just need to show one more thing. Note that the AJS/Billey formula produces an element of the $n$-dimensional torus $T$-equivariant cohomology $H_T^*(V)\cong \Z[y_1,...,y_n]$. However, the theorem from \cite{Rossmann} involves $H_{\C^\times}^*(V)\cong \Z[h]$. Thus, our final step will be to produce a map $H_T^*(V) \to H_{\C^\times}^*(V)$. We will show that this map takes the rightmost term 
$\sum \prod_{r\in R}  \hat{\beta_r}$ and maps it to $(\text{\# pipe dreams on } U_\lambda)\cdot h \in H_{\C^\times}^*(V)\cong \Z[h]$. Then the proof will be complete.

\subsubsection{Equation $(*)$ Part (4)}

Finally, it remains to be shown, under the map from $H_T^*(U_\lambda)$ to $H_{\C^\times}^*(U_\lambda)$, that the right side of the AJS/Billey formula corresponds to the number of pipe dreams.

Recall that we showed above the isomorphisms $H_T^*(U_\lambda) \cong H_T^*(pt) \cong H_T^*(X_\circ^{v(\lambda)})$. However, to use Rossmann's Theorem \ref{thm:Rossmann}, we need to have $H_T^*(V)$ for a vector space $V$. We will get this via an isomorphism of $U_\lambda$ with the vector space $M_{k,n-k}$, the space of $k$-by-($n-k$) matrices. This is depicted in the following diagram, which shows $M_{k,n-k}$ being identified with another space of matrices with the identity in columns $\lambda$, which we can call $(M_{\lambda})_{k,n-k}$; the row span of $(M_{\lambda})_{k,n-k}$ is equal to $U_\lambda \subseteq Gr(k,n)$. 

\begin{tikzcd}
M_{k,n-k} \ar[r,"\sim"] & \left\{ \mqty(&1&0&\\ *&0&1&*\\&\vdots&\vdots&) \right\} \ar[rr,"\sim"]  \ar[drr,shorten >=3ex,"row span"'] & & U_\lambda \curvearrowleft T \hookleftarrow \C^\times \ar[d,hookrightarrow, shift right=5ex]  \\
\left[ v_1\cdots v_{n-k}\right] \ar[u,"\rin"] \ar[r,mapsto] & \mqty(&&1&0&\\ v_1 & v_2 &0&1& \ar[u,"\rin"] \hdots \\&&\vdots&\vdots&) & & Gr(k,n)
\end{tikzcd}

We also recall that the right side of AJS/Billey is: $\sum_{R\subseteq Q, R\text{ reduced}} \prod_{r\in R}  \hat{\beta_r}$, where $Q$ corresponds to the word representing $v(\lambda)$, and then the sum is over all $R$, which are words representing the siteswap $f$. In other words, there is precisely one summand in this formula for each pipe dream representative of $\Pi_f \cap U_\lambda$. We will show that our choice of map $H_T^*(V) \to H_{\C^\times}^*(V)$ takes each of these summands to $1\cdot h^{codimX} \in H_{\C^\times}^*(V)\cong \Z[h]$, so the result will be $(\text{\# pipe dreams of }\Pi_f \text{ on } U_\lambda)\cdot h^{codimX}$. Then we will be done.

The map $H_T^*(V)$ to $H_{\C^\times}^*(V)$ is entirely determined by the map $\C^\times \to T$. We describe the relevant maps explicitly in a moment, but first we give the general theory. 

Suppose we have a map of tori $S\to T$. Then we get the corresponding Lie algebra map $\mf{s} \to \mf{t}$ as well as the dual map $\mf{s}^* \leftarrow \mf{t}^*$. Then we have the following commutative diagram:

$$\begin{tikzcd}
\mf{t}^* \ar[d,hookrightarrow]  \ar[r] & \mf{s}^* \ar[d,hookrightarrow] \\
Sym(\mf{t}^*) \ar[d,leftrightarrow,"\cong"]\  \ar[r] &  Sym(\mf{s}^*) \ar[d,leftrightarrow,"\cong"]\\
H_T^* \ar[r] & H_S^*
\end{tikzcd}
$$

Note that if we pick a basis for $\mf{t}^*$, then $Sym(\mf{t}^*)$ becomes a polynomial algebra, and these basis elements are degree-2 generators (the cohomology is in even degree). Also note that $H_T^* := H_T^*(pt) = H_T^*(V)$ since a vector space $V$ is contractible. Recall that an application of Theorem \ref{thm:Rossmann} required that $\C^\times$ acts with all weights 1. 

\begin{Lem}
We satisfy the condition of the circle $\C^\times \subseteq T$ acting on $U_\lambda$ with all weights 1 due to our choice of map $\C^\times \to T$, which is: \[z\mapsto (z,1,1,z,...,1)\] with the entries $z$ in the coordinates of $[n]\backslash \lambda$.
\end{Lem}

\begin{Exa}
Suppose we are in $Gr(2,4)$ with $\lambda=\{1,3\}$. Then we can represent $U_\lambda$, and the action of $\C^\times$ on the right as:
\[\mqty(1 & * & 0 & * \\ 0 & * & 1 & *) \mqty(\dmat{1,z,1,z}).\]
This scales all the $*$'s, which give coordinates of $U_\lambda \cong \mathbb{A}^{k(n-k)}=\mathbb{A}^4$, by $z^1$. Thus, $\C^\times$ acts via dilation on $U_\lambda$, as required.

Alternatively, we can consider the full action of $T$ on $U_\lambda$ and find the action of $\C^\times$ inside $T$ using:
\[\mqty(\dmat{z_1^{-1},z_3^{-1}}) \mqty(1 & * & 0 & * \\ 0 & * & 1 & *) \mqty(\dmat{z_1,z_2,z_3,z_4}) \] \[ = \mqty(\dmat{z_1^{-1},z_3^{-1}}) \mqty(z_1 & z_2* & 0 & z_4* \\ 0 & z_2* & z_3 & z_4*) = \mqty(1 & \frac{z_2}{z_1}* & 0 & \frac{z_4}{z_1}* \\ 0 & \frac{z_2}{z_3}* & 1 & \frac{z_4}{z_3}*).\]

We see that in order to maintain our matrix representative in the form $\mqty(1 & * & 0 & * \\ 0 & * & 1 & *)$ we must ``unscale" on the left, which results in the coordinates having factors $\frac{z_{col}}{z_{row}}$, where the number $row$ refers to: within this number's row, find where the $1$ is, and record the column of this $1$.
\end{Exa}

\begin{proof}
Note that for a general embedding of $\C^\times $ in $T$ acting with arbitrary weight, we have $z_i=z^{y_i}$. Then $\frac{z_{col}}{z_{row}}= z^{y_{col}-y_{row}}$. Since we want the action to be by dilation, we need all these $y_{col}-y_{row}=1$. We just saw that we could do this via the map $z\mapsto diag(z,1,1,z,...,1)=diag(z^1,z^0,z^0,z^1,...,z^0)$ with the $z$'s in the non-$\lambda$ spots. In other words, by setting \\
\[y_i= \begin{cases} 0 & i\in \lambda \\ 1 & i\not\in \lambda \end{cases},\]
then $y_{col}-y_{row}=1-0=1$ since the $col$ corresponds to the columns not in $\lambda$ and the $row$ correspond to the unit vectors in the matrix which are where $\lambda$ is. 
\end{proof}

If the $T$-action on the tangent spaces to the fixed points on a flag manifold $G/P$ contains dilation, $G/P$ is called cominuscule. We have just recovered the fact that Grassmannians are cominuscule (something not true for full flag manifolds).

Let's now write the remaining maps explicitly.
$\C^\times \to T, z\mapsto diag(1,z,z,1,....,1,z)$ where the $z$'s exist precisely in the places in the diagonal where $\lambda$ is not gives us:
$\mf{s} \to \mf{t}, 1 \mapsto [1,0,0,1,...,0]$ the map of Lie algebras, where we have just taken the derivative of the map above and extracted the diagonal.

\[ \mf{s}^* \leftarrow \mf{t}^* : \mqty(1\\0\\0\\1\\ \vdots \\ 0)\]

Finally, once we go to $Sym(\mf{t}^*) \to Sym(\mf{s}^*)$ this is just a map of polynomial rings, where this map $\mqty(1 & 0 & 0 & 1& \cdots & 0)^T$ sends all variables $y_i$ in $Sym(\mf{t}^*)$ corresponding to the numbers in $\lambda$ to $0$, and all other variables to $h$. 

Recall Proposition \ref{prop:AJSrowcolform}, which noted that 
\[ [\Pi_f \cap U_\lambda]_T = \sum_{\substack{R\subseteq v(\lambda) \\ \prod R=v(f) \\ R \text{ reduced}}} \prod_{r\in R}  [(y_i)_r - (y_j)_r] \] where $i$ is the column and $j$ is the row of a cross in a pipe dream.

Now, we can apply to this our choice of map $\tilde{\lambda}:\C^\times \to T$ sets all $y_{row}=0, y_{col}=h\in \Z[h]$, so since each summand in the AJS/Billey formula is a product of $codim(\Pi_f)$ many $\beta$'s, each summand will be $1\cdot h^{codim \Pi_f}$. As we already showed above, there is a single summand for each pipe dream representative of $\Pi_f$ on $U_\lambda$, so the AJS/Billey formula is equal to $(\text{\# pipe dreams of }\Pi_f \text{ on } U_\lambda)\cdot h^{codim \Pi_f}$. This completes the proof of the uniqueness (smoothness) part of Theorem \ref{thm:smoothone}.

Finally for the proof of the existence part, that $\lambda \in \Pi_f$ is equivalent to there existing at least one affine pipe dream, this follows from the fact that each affine pipe dream for $\Pi_f$ on $U_\lambda$ is a subword for $\Pi_f$ in the word for $\lambda$ in accordance with what we have already described in Section \ref{section:affinepipe}. In addition, it is a general fact of words in Bruhat order that a point $\lambda$ being contained in a Schubert variety defined by $f$ is equivalent to the word $v(\lambda) > v(f)$ in Bruhat order, and furthermore, this is equivalent to $v(f)$ being a subword in a reduced for for $v(\lambda)$.

\subsection{The Ordering and Smoothness}

We presented the following figure in the introduction. In the ordering (given by deletion and contraction) of positroid varieties and points, the smooth versus singular pairs can be divided in the following way:

 \begin{figure}[htbp]
	\includegraphics[scale=0.5,clip=true]{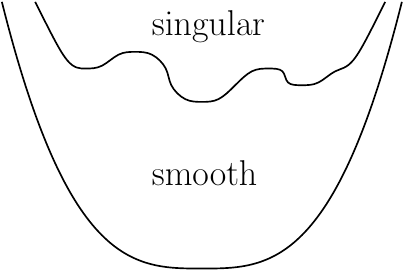}
\end{figure}

In other words, smooth positroid varieties remain smooth under deletion and contraction, while singular positroid varieties if obtained from another positroid variety via deletion or contraction must have come from a singular positroid variety. We can state this in the following proposition:

\begin{Prop}\label{prop:delsmooth}
If $\lambda$ is smooth on $\Pi_f$, then $del_i(\lambda)$ is smooth on $\Pi_{del_i(f)}$ and $contr_i(\lambda)$ is smooth on $\Pi_{contr_i(f)}$. Or this can be stated contrapositively: if $del_i(\lambda)$ is singular on $\Pi_{del_i(f)}$ or $contr_i(\lambda)$ is singular on $\Pi_{contr_i(f)}$, then $\lambda$ is singular on $\Pi_f$.
\end{Prop}

\begin{proof}[Geometric proof]

By e.g. \cite{BilBra03}, if the group $G$ is reductive, then whenever $G$ acts on a regular set $X$, the fixed point set $X^G$ is also regular. In our case, $G=T$ is the torus, so the conclusion follows. 

In particular, if $\Pi_f$ is smooth at $p$, then $\Pi_f^{\C^\times}$ is smooth at $p$, and $\Pi_f^{\C^\times}=(\Pi_f\cap \Pi_{del,i}) \cup (\Pi_f\cap \Pi_{contr,i}) = Del_i(\Pi_f) \cup Contr_i(\Pi_f)$. 
\end{proof}

We can give a combinatorial proof of this proposition using the machinery of affine pipe dreams:

\begin{proof}[Combinatorial proof]

Without loss of generality, let us do deletion on the $i$th column in a rigid pipe dream (the case of contraction will be analogous but for a row). 

The column will alternate between consecutive sets of crosses and consecutive sets of elbows. It will be helpful for the proof to look at the following figure which zooms in on the $i$th column, where the consecutive sets, for simplicity, have been drawn as groups of 3:

 \begin{figure}[htbp] \centering
	\includegraphics[scale=0.7,clip=true]{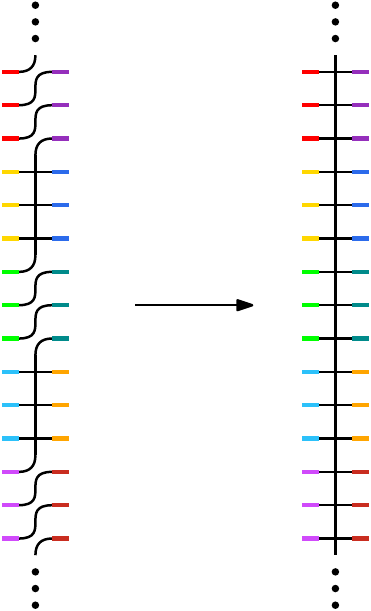}
\end{figure}

Note that the colors on the two pipe dreams in the figure are the same; for each one, reading from top to bottom, the left side (of column $i$) goes: light red - light orange - light green - light blue - light purple; and on the right side: dark purple - dark blue - dark green - dark orange - dark red. The fact that we have flipped the order of the coloring (while doing dark instead of light to distinguish them) is intentional: the pipe dream has a northwest-southeast symmetry, coming from the fact that the pipes (and the moves) go southwest to northeast. Thus, it is sufficient to make arguments on the just left side (light) colors, so for ease of reference, we call the left/light colors just: red, orange, green, blue, purple; whereas if we need to distinguish the colors on the right side, we will call them: dark purple, dark blue, cyan (or dark green), dark orange, dark red. The green to cyan pipes in the middle are where we will focus. After all, only the elbow tiles are changed in the deletion. 

Therefore, what we are doing here is picking some random block of consecutive elbow tiles in the column, and showing that, even though those tiles are changed (turned into crosses), no new moves can be created. If we can show this for the green to cyan tiles in this diagram, then this applies to any other elbow tile in this column that gets deleted, and the proof will be complete.

If we need to further distinguish the pipes within each block, we can call them, for example: red 1, red 2, red 3, going from top to bottom so that red 1 is the highest. 

We first make several observations. 

(1) There can be no crossing among any of the green pipes anywhere because they are elbows/near-misses along the $i$th column; any crossing would then provide a move (specifically, the cross could then be moved to the $i$th column), contradicting rigidity. For the same reason, there can be no crossing among the cyan pipes. (1a) Note, however, that the orange pipes have no such restriction among themselves.

(2) The green pipes cannot cross any of the orange pipes. The reason is that, by (1), in order for any of the green pipes to cross any of the orange pipes, green 1 would have to cross (first). However, since green 1 crosses all of the orange pipes along column $i$, this would create a double crossing. 

(3) The green pipes cannot cross any of the red pipes. The reason is that, by (1), in order for any of the green pipes to cross any of the red pipes, green 1 would have to cross red 3 (first). However, since green 1 and red 3 have a near miss, this would create a move, contradicting rigidity. The same argument show that the purple pipes cannot cross any of the green pipes. 

(4) The green pipes \textit{can} cross the blue pipes. 

(5) By symmetry, all these points are valid for the right side of column $i$ as well: just replace ``color" by ``dark color."

We can recap all of this succinctly. The only crossings among pipes that are allowed are: (I) \textit{within} each of block pipes coming from consecutive crosses along the $i$th column, and (II) between a block of consecutive elbows and the block of consecutive crosses directly below them (on the left side of column $i$; or, on the right side, by symmetry, between consecutive elbows and consecutive crosses directly above)

We now show that doing the deletion does not create any new moves. 

The deletion only changes the connectivity of the elbow tiles. In other words, the orange and blue pipes do not change, so there cannot be any new moves coming from the orange or blue pipes. Let us focus in on the green pipes, and ask the question of whether there can be any new moves coming from the change in the connectivity of the green pipes.

A move consists of 2 pipes having a cross and a near-miss. Obviously, many new crosses are created along column $i$; however, these crosses cannot participate in any move because the $i$th column pipe runs straight up and down vertically, crossing every pipe along the $i$th column, and thus has no near-misses. Therefore, a new move must come from a crossing of 2 pipes on the left side of column $i$, for which these 2 same pipes have a near-miss on the right side of column $i$ (or vice-versa switching left and right, but it is sufficient to just consider the former case). Since this is a new move, it did not exist before.  The cross and near-miss cannot have both been on the same side (left or right); otherwise, this move would have existed prior to the deletion and thus not only not be a new move, but the original pipe dream would have had a move, contradicting rigidity.

Thus, the question now becomes, (a) what possible crosses involving the green pipes on the left are there, and (b) what near-misses involving these pipes are there on the right. The work we did above now comes in handy. We know that the green pipes cannot cross among themselves by (1), and in fact cannot cross any pipes except the blue ones by (II). This answers (a). Now we consider the question of (b). Since the green pipes can only potentially cross the blue pipes, the question is whether it is possible for a potential green-blue crossing on the left, after we do the deletion which causes these pipes to turn into dark green - dark orange pipes on the right, to involve a dark green - dark orange near-miss on the right.

This is not possible for the following reason. Pipe dark green 3 crosses all three of the dark orange pipes. Since the pipe dream is rigid, that means that dark green 3 cannot have a near miss with any of the dark orange pipes. By (1) above, none of the dark green pipes can cross each other, which mean that none of the dark green pipes can have a near miss with any of the dark orange pipes. This completes the proof.

\end{proof}

However, our work above allows us to say something further about this last point. In order for dark green 3 to not have a near miss with any of the dark orange pipes, we must have another color pipe come between the dark green and dark orange pipes. We saw above in (4) or (I) that, on the left, the elbow pipes could cross with the cross pipes in the block below, which, on the right, corresponds to elbow pipes crossing with the cross pipes above. This means that one, or any (since the dark blue can cross among themselves by (1a)), of the dark blue pipes can come down by crossing the dark green pipes to create separation from the dark orange pipes; or, dark red pipe 1 would have to come up (the other dark red could also come up, but dark red pipe 1 would have to come up first by (1)).

Thus, if there were no moves prior to the deletion, then there are no moves after the deletion. The contrapositive statement is that a singular positroid variety $\Pi_{f'}$ coming from another positroid variety $\Pi_f$ via deletion or contraction implies that $\Pi_f$ is itself singular, since a smooth positroid variety being deleted to a singular positroid variety is impossible.

Finally, we note here that we can give another name to going down in this ordering, a concept which can also be applied to (affine) pipe dreams.

\begin{Def}\label{def:dcreduce}
(1) We say that $(\Pi_{f_1},\lambda_1)$ \textbf{d/c-reduces} (for deletion/contraction) to $(\Pi_{f_2},\lambda_2)$ if a series of deletions and contractions of $(\Pi_{f_1},\lambda_1)$ results in $(\Pi_{f_2},\lambda_2)$.

(2) The notion of ``d/c-reduces" in particular applies to (affine) pipe dreams: Let's denote the set of affine pipe dreams corresponding to $(\Pi_{f_1},\lambda_1)$ by: $PD_{(\Pi_f,\lambda)}$. This becomes a bijection if we consider the elements in $(\Pi_{f_1},\lambda_1)$ up to equivalence by the ``moves" defined in Definition \ref{def:moves}. Furthermore, we have a notion of deletion and contraction for individual pipe dreams, that is, applying deletion and contraction to elements $\delta \in PD_{(\Pi_f,\lambda)}$ (see the end of Xection \ref{section:affinepipe}).

Thus, we can consider, more specifically, whether a given $\delta_1 \in PD_{(\Pi_{f_1},\lambda_1)}$ goes to $\delta_2 \in PD_{(\Pi_{f_2},\lambda_2)}$ under a series of deletions and contractions on affine pipe dreams described in Section \ref{section:affinepipe}. Obviously, if we consider these affine pipe dreams only up to equivalence under moves, then the relations will be the same as (1) above. However, since it will be easier to work with individual $\delta \in PD_{(\Pi_f,\lambda)}$ in many of the following proofs, in the rest of this paper when we use the term ``d/c-reduces," we will use it to refer to this latter notion (2). 
\end{Def}

\section{Proof of Main Theorem}

Recall the following figure from Section \ref{sub:theresult}, which displays an example of maximal rectangles. Maximal rectangles are defined simply as rectangles within the pipe dream whose southwest and northeast are also corners of the pipe dream shape.

\begin{figure}[htbp]
	\includegraphics[scale=0.5,clip=true]{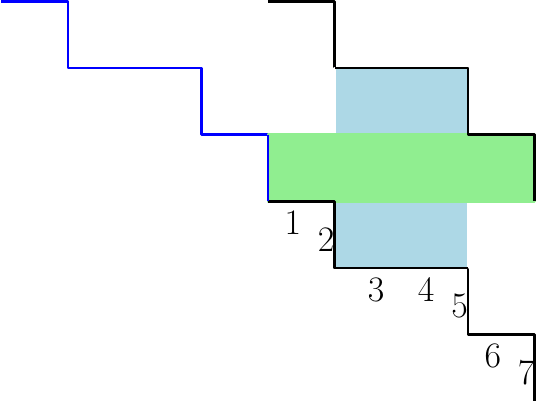}
\end{figure}

\begin{Def}
A pipe dream in a rectangular shape \textbf{reduces to SE/NW partitions} if, after deleting all entire rows and columns of crosses, the result is a partition in the southeast and northwest corners. 

This is equivalent to the condition that all crosses must be contained in: (1) entire rows of crosses, (2) entire columns of crosses, (3) a partition in the northwest, (4) a partition in the southeast.
\end{Def}

\begin{Exa}\label{exa:SENWpart}
The following is an example of a rectangular pipe dream that reduces to SE/NW partitions:
 \begin{figure}[htbp]
	\includegraphics[scale=0.5,clip=true]{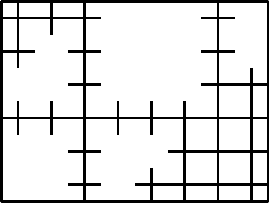}
\end{figure}
\end{Exa}

\begin{Thm}\label{thm:mainthm} [Main Theorem]
Let $\delta \in PD_{(\Pi_f,\lambda)}$ be an affine pipe dream for $\Pi_f \cap U_\lambda$. If each maximal rectangle in $\delta$ reduces to NW/SE partitions, then $\Pi_f$ is smooth at $\lambda$; otherwise, it is singular. 
\end{Thm}

\begin{proof}

This theorem consists of showing two parts:\\ 
(1) Showing that smoothness can be determined just by looking at the maximal rectangles in an affine pipe dream. This is proven in Proposition \ref{prop:insidemaxrect} below. \\
(2) Showing that the thing that needs to be looked at on each maximal rectangle is that each maximal rectangle is a NW/SE partition. This is proven in Theorem \ref{thm:NW/SE} below. 

However, first we give the full proof, referencing both Proposition \ref{prop:insidemaxrect} and Theorem \ref{thm:NW/SE} as needed. 
\\

We want to show that smoothness of $(\Pi_f,\lambda)$ is equivalent to $\delta$ having no moves \textit{in any maximal rectangle}. But we already have Theorem \ref{thm:smoothone}, which says that smoothness of $(\Pi_f,\lambda)$ is equivalent to its pipe dream having no moves (anywhere), and, by Proposition \ref{prop:insidemaxrect}, all moves exist within maximal rectangles. Thus, if all maximal rectangles have no moves, then the pipe dream for $(\Pi_f,\lambda)$ has no moves. The other direction, that the pipe dream for $(\Pi_f,\lambda)$ having no moves implies that all its maximal rectangles have no moves, is automatic. 

Therefore, smoothness of $(\Pi_f,\lambda)$ is equivalent to $\delta$ having no moves in any maximal rectangle.
Thus, we just have to show that rigid rectangles reduce to NW/SE partitions and the proof will be complete, and this is shown in Theorem \ref{thm:NW/SE}.
\end{proof}

\begin{Prop}\label{prop:insidemaxrect}
Every cross-elbow move exists within some maximal rectangle. More precisely, every pair consisting of a cross and a near-miss where this cross can move is contained within a maximal rectangle. 

Trivially, this implies that the portions of the two pipes involved in the move, starting from the cross and running all the way to the near-miss, are entirely contained in this maximal rectangle.
\end{Prop}
In fact, the proof below shows the stronger statement that if we consider a cross and the set of all places where it can move, these will all be contained in a maximal rectangle.  
\begin{proof}
The very notion itself of nonrigidity is that there is a cross to the southwest and a near-miss of the pipes in that cross somewhere to the northeast (or vice-versa with southwest and northeast switched). In the sorts of shapes given by affine pipe dreams we are considering, such a cross and then near-miss combination must exist within a full rectangle because all the jagged (or convex) parts of the shape are in the northeast and southwest.

This is depicted in the red rectangle in the following figure: 

\centerline{
	\includegraphics[scale=0.5,clip=true]{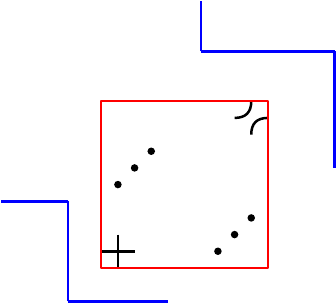}
}
\end{proof}

Next, in order to prove Theorem \ref{thm:NW/SE}, we will first need the following proposition, which itself requires a lemma:

\begin{Lem}\label{lem:fillinrect}
In a reduced pipe dream, any time there is an elbow followed by $m$ consecutive crosses to its east and $n$ to its south,  then the $n\times m$ box formed by having the easternmost and southernmost crosses as corners must be entirely filled in with crosses. The same is true with ``east and south" replaced by ``west and north." That is: 

 \begin{figure}[htbp] \centering
	\includegraphics[scale=0.5,clip=true]{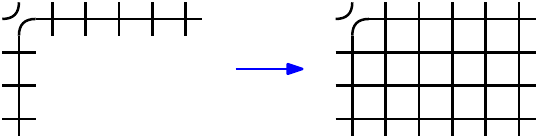}
\end{figure}
\end{Lem}

\begin{proof}
The following figure provides the proof of this lemma.
 \begin{figure}[htbp] \centering
	\includegraphics[scale=0.5,clip=true]{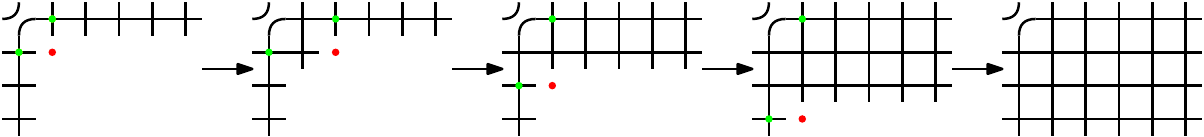}
\end{figure}
Consider the red dot in the first figure on the very left. It cannot be an elbow, or else the two green dots would be a double crossing (non-reduced); so the red dot must be a cross. So once we have filled that tile in with a cross, move the red dot to the right, where it again must be a cross, by a similar argument. We continue inductively until the entire first row is filled in with crosses. Then we move on to the 2nd row, and apply the same argument to fill it in with crosses. We continue on, row by row, until the $n\times m$ rectangle is filled in with crosses. 
\end{proof}

\begin{Prop}\label{prop:Le}
A rectangular pipe dream is rigid if and only if: (1) no cross has both an elbow to its north as well as to its east, and (2) no cross has both an elbow to its south as well as to its west.

\end{Prop}

This is essentially a ``double-Le condition," and the proof is similar to statements in Section 19 of \cite{Post} or Section 4.3.3 of \cite{Sni}. As hinted in \cite{Sni}, one of these conditions basically corresponds to the bottom pipe dream (see Appendix B), so the other would correspond to the top pipe dream.

\begin{proof}

First we prove that if a pipe dream satisfies this double-Le condition, then it is rigid. We show the contrapositive, that if a pipe dream is not rigid, then it must violate this condition. The proof is essentially the following picture:
\\

 \begin{figure}[htbp] \centering
	\includegraphics[scale=0.5,clip=true]{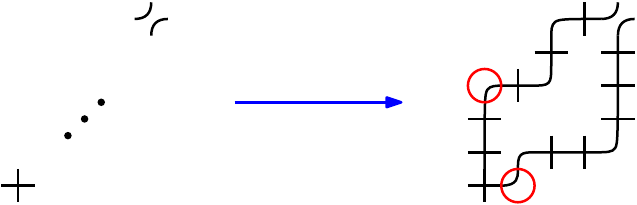}
\end{figure}

Let's assume that this pipe dream is not rigid. This means that, as shown at the top of the figure, there exists two pipes that cross somewhere and later have a near-miss, either to the northeast of the cross as in the figure, or to the southwest (not depicted, but analogous). This requires the two pipes to each have a ``first bend," circled in red in the representative example bottom left figure, because if they never bent, then they would continue on indefinitely north and east and never cross. Thus, a pipe dream with a move either contains a cross with elbows both to its north and east (as shown in this figure), thus violating the north/east Le condition, or it has a move with the figure reflected across the $x=-y$ line; that is, with a violation of the south/west Le condition. 
\\

\minip{.77}{12mm}{ Now we prove the other direction: if a pipe dream violates this double-Le condition, then it must not be rigid. Without loss of generality, let's again assume that it violates the north/east Le condition (the south/west case is analogous). This means that there is a cross with elbows somewhere to its north and east, and we can then assume the elbows to be the closest in these directions. That is, we have:
}
\hfill
\minip{.20}{0mm}{
	\includegraphics[scale=0.4,clip=true]{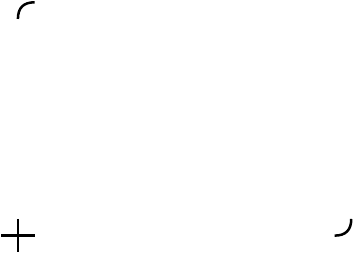}
}
\vskip .05in

\minip{.77}{12mm}{ Furthermore, once they bend at an elbow, they cannot keep moving along straight lines, as depicted in the dotted lines in the following figure, or else they would cross again, which is violation of reducedness:
}
\hfill
\minip{.20}{0mm}{
	\includegraphics[scale=0.53,clip=true]{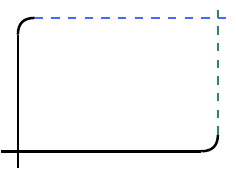}
}
\vskip .05in

\minip{.77}{12mm}{Because of this, at least one of the pipes must have another bend. In fact, we must have each pipe bending a 2nd time because otherwise we would be able to apply Lemma \ref{lem:fillinrect} to this rectangular region of interest to arrive at a nonreduced pipe dream, like the following figure: 
}
\hfill
\minip{.20}{0mm}{
	\includegraphics[scale=0.53,clip=true]{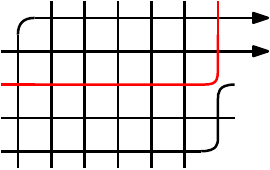}
}
\vskip .05in

\minip{.77}{12mm}{ Thus, putting in these additional bends, and filling in the parts to the west and south of the bends with crosses via an application of Lemma \ref{lem:fillinrect}), the pipe dream looks something like:
}
\hfill
\minip{.20}{0mm}{
	\includegraphics[scale=0.5,clip=true]{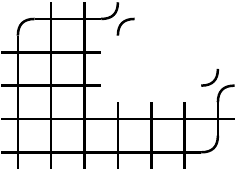}
}
\vskip .05in

In other words, by imposing a violation of the double Le condition and reducedness (simply having a valid pipe dream), we have concluded that we have two pipes that cross, that each bend, and then bend again, and everything to the southwest of the rectangle formed by the places where their second bends are must be filled in with crosses.
\\

\minip{.77}{12mm}{ Next, we apply the assumption of rigidity we made for the sake of contradiction. In the following figure, if the elbow on the right were to bend inwards, then we would have the red elbow depicted, and there would be the possible move depicted by the two red dots, which is not allowed by rigidity:
}
\hfill
\minip{.20}{0mm}{
	\includegraphics[scale=0.5,clip=true]{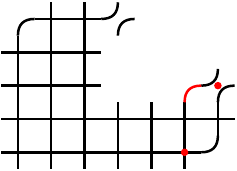}
}
\vskip .05in

\minip{.77}{12mm}{ So we must have the green cross below. But then, inductively, we can make the same argument for the spot next to it, for if it was an elbow, then we would have the possible move with the two red dots:
}
\hfill
\minip{.20}{0mm}{
	\includegraphics[scale=0.5,clip=true]{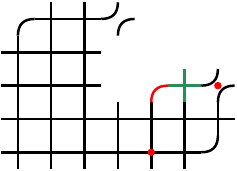}
}
\vskip .05in

\minip{.77}{12mm}{ Thus, it is clear that by this argument, we must fill in everything along the same lines as the elbows with crosses, to get the following:
}
\hfill
\minip{.20}{0mm}{
	\includegraphics[scale=0.5,clip=true]{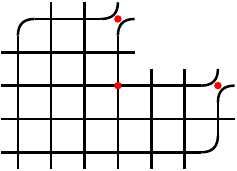}
}
\vskip .05in

\minip{.77}{12mm}{ If we look at the red dots, we see that the configuration that we started with (with two northwest southeast elbows and a cross at one of the corners) is here, only smaller. We can keep doing this inductively, but due to the finite size, we cannot go on forever. We must reach a point where we have a move, giving us the contradiction:
}
\hfill
\minip{.20}{0mm}{
	\includegraphics[scale=0.5,clip=true]{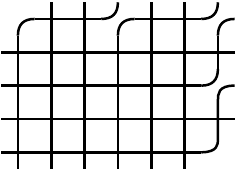}
}

\end{proof}

\begin{Cor}
If a pipe dream is rigid within a rectangular shape, then every time there is a cross with an elbow to its south, then it must have crosses extending all the way west. We can visualize it as follows:

 \begin{figure}[htbp] \centering
	\includegraphics[scale=0.35,clip=true]{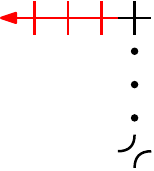}
\end{figure}

The same statement is true with ``west" and ``south" reversed, as well as replaced with ``north" and ``east". 
\end{Cor}

We note that the previous Prop \ref{prop:Le} was only valid within rectangular shapes. In nonrectangular shapes, a violation of the Le condition is not sufficient to conclude nonrigidity. A counterexample is provided by the following figure, where the two red dots represent a Le-violating condition, yet the portion of this pipe dream depicted is perfectly rigid. 
 \begin{figure}[htbp] \centering
	\includegraphics[scale=0.5,clip=true]{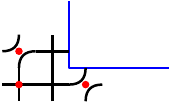}
\end{figure}

The following corollary, however, gives a condition for rigidity in any pipe dream shape (not just rectangular ones).

\begin{Cor}\label{cor:box} [of the proof]

An affine pipe dream is rigid if and only the following two types of rectangular pipe crossings do not appear anywhere:

 \begin{figure}[htbp] \centering
	\includegraphics[scale=0.5,clip=true]{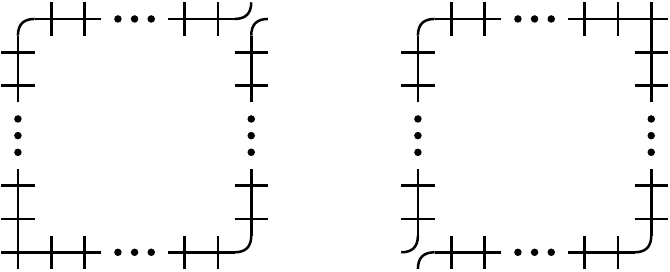}
\end{figure}
Note that the arguments above about reducedness above imply that the inside of these rectangles is filled in completely with crosses. 
\end{Cor}

{\em Remark 1.} This corollary does not state that in a nonrigid \textit{pipe dream}, that every nonrigid \textit{pair of pipes} must be of this form. In fact, two pipes that cross and bend inwards at some point before having a near-miss is indeed possible; however, even in this case, the nonrigid pipe dream must contain nonrigid pairs of pipes of the form depicted in this corollary. The following figure gives an example, where the red and blue types have a move that is not of the type stated in the corollary, but there are 3 pairs of pipes that do have a move of the stated type, which are depicted in the 3 pairs of (different shades of) green dots: 

 \begin{figure}[htbp] \centering
	\includegraphics[scale=0.6,clip=true]{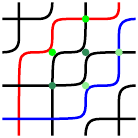}
\end{figure}

{\em Remark 2.} First, we recall that the main theorem gives a criterion for testing smoothness (equivalent to rigidity) which is checkable on maximal rectangles. The fact that we only need to check on maximal rectangles is evident from this corollary, since this sort of pipe dream configuration can only exist within rectangular shapes. Of course, this is just a special case of Proposition \ref{prop:insidemaxrect}.

\begin{proof}
The proof of Corollary \ref{cor:box} is already in the proof of the Proposition \ref{prop:Le}. One direction is obvious: a box configuration as depicted implies nonrigidity. Conversely, if we have something that is nonrigid, then we have a cross and then a near-miss for a pair of pipes, like in the previous figure. Then these pipes must bend somewhere. The south and west portions of these bends must be filled in with crosses via the Lemma, as explained in the previous proof. Then, if there are any elbows in the remaining portion, they will form a rectangular configuration like that depicted. If not, then we arrive at a smaller box, as explained previously. 
\end{proof}

Note that this rule, that neither of the two box configurations above occur, encompasses the sorts of moves considered in \cite{BB93}. Specifically, they considered moves of the following form:

 \begin{figure}[htbp] \centering
	\includegraphics[scale=0.5,clip=true]{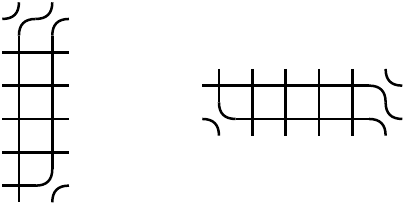}
\end{figure}

These can be viewed as size $(1,j)$ or $(j,1)$ boxes of the form in the previous corollary. 

\begin{Thm}\label{thm:NW/SE}
All rigid rectangles reduce to NW/SE partitions.
\end{Thm}

\begin{proof}

What we will prove below is that every cross (schematically shown by circling the crosses in red) is contained in one of the four configurations below.

\begin{figure}[htbp]\centering
	\includegraphics[scale=0.5,clip=true]{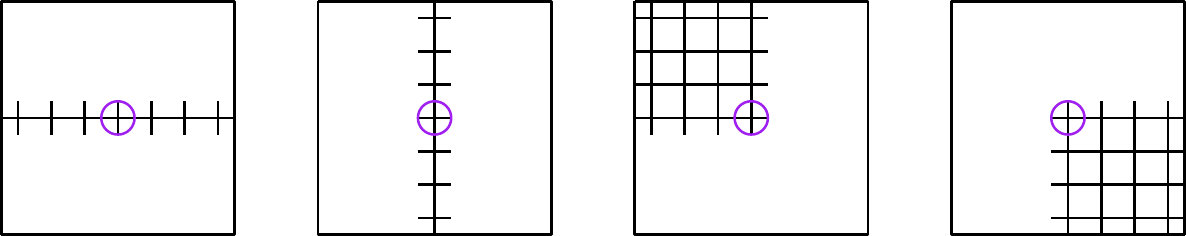}
\end{figure}

The proof is mostly constructive, and proceeds by going through the following diagram:

  \begin{figure}[htbp]\label{thegrid}\centering
	\includegraphics[scale=0.45,clip=true]{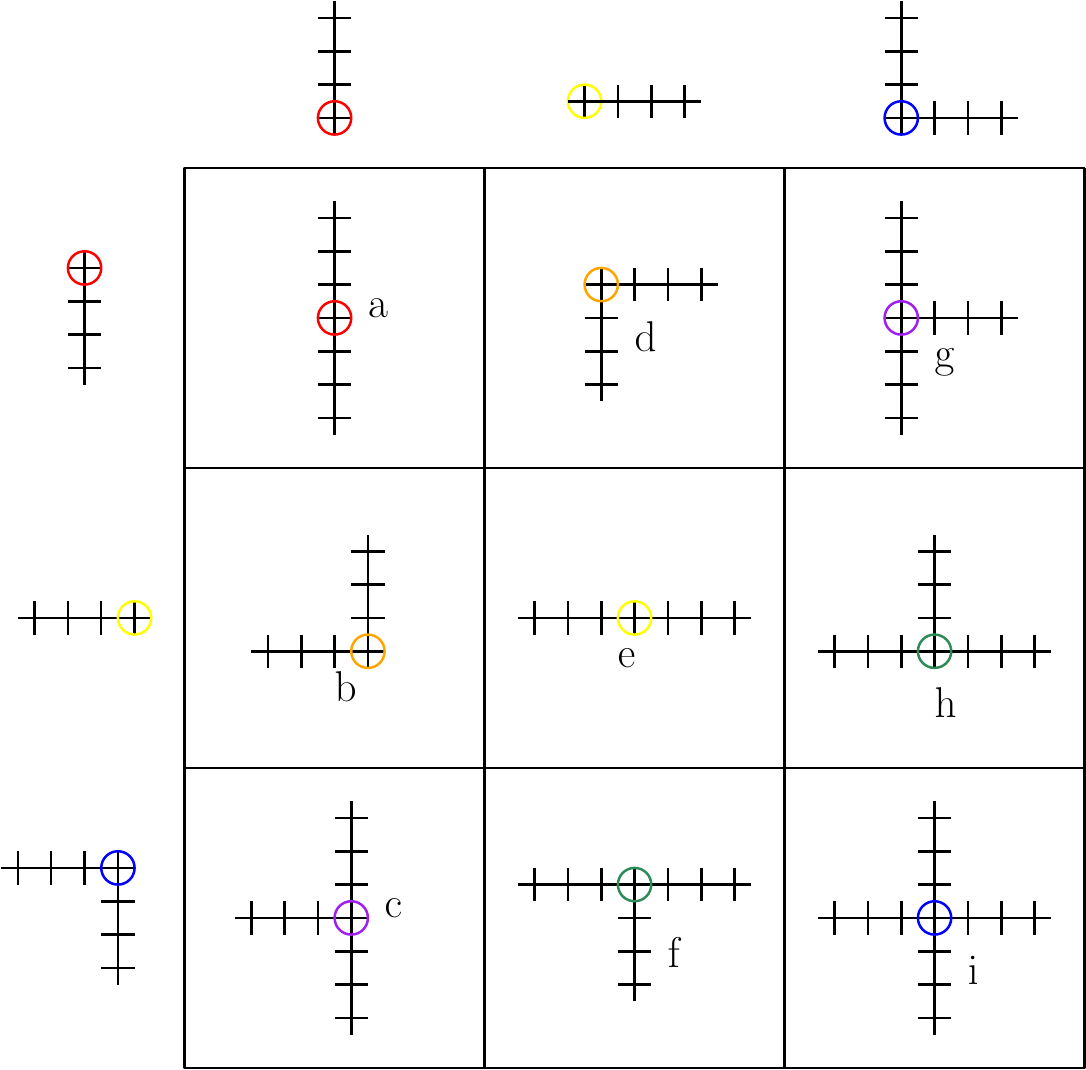}
	\caption{Cases of the theorem}
\end{figure}

\newpage

We can think of these as two steps, which we can call NE and SW. These come from the previous theorem which shows that any cross within a rigid diagram cannot have elbows both to its north and east, or both to its south and west. Thus, to ensure that both the NE condition and SW condition are satisfied, we can think of doing 2 steps: ensuring that the NE condition is satisfied, and ensuring that the SW condition is satisfied. 
\\

Looking at the diagram, we start with a single cross, which we mark by giving it a circle. The top (separating the columns) then shows the result of applying the NE condition to this cross. Any time we have a series of three crosses coming out from the circled cross, that pictorially depicts an unbroken line of crosses stretching all the way to the wall. Thus, for example, the leftmost configuration at the top shows the cross (circled in red), with a line of crosses stretching north all the way to the wall. This satisfies the NE condition, since it can no longer have elbows \textit{both} to its north and east. If we then apply the SW step to this configuration, then we obtain diagrams (a), (b) and (c). Diagram (a) is obtained by satisfying the SW rule by extending south, (b) by ending west, and (c) by extending both south and west. Diagrams (d)-(f), and (g)-(i) follow the same pattern of elongation starting from their diagram satisfying the SW condition.
\\

Now, each (a), (e), and (i) is already contained in an entire row or column of crosses (or union thereof), so they already satisfy the conditions of the theorem. Thus, we have to deal with the rest. We do the argument for two of these: (c) and (d). The others follow by analogous arguments, and are left to the reader.
\\

\minip{.77}{12mm}{ (c) looks like the following:
}
\hfill
\minip{.20}{0mm}{
  \includegraphics[scale=0.5,clip=true]{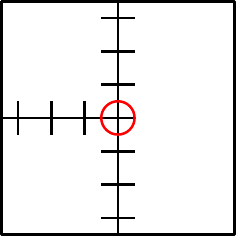}
}
\vskip .05in

\minip{.77}{12mm}{ The crosses along the vertical line all satisfy both NE and SW conditions. However, note that the empty space is unknown: at the moment, it can be either crosses or elbows. Thus, if we look at the crosses boxed in grey in the next figure, they do not satisfy the NE condition if the empty space is all elbows. 
}
\hfill
\minip{.20}{0mm}{
  \includegraphics[scale=0.5,clip=true]{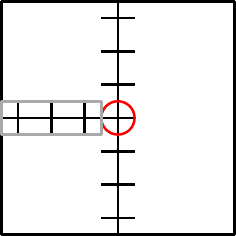}
}
\vskip .05in

\minip{.77}{12mm}{ In order to have these crosses satisfy the NE condition, we can extend them east, or extend them north. If we extend them east, then this extra set of crosses is sufficient for satisfying the NE condition for all the crosses boxed in grey:
}
\hfill
\minip{.20}{0mm}{
  \includegraphics[scale=0.5,clip=true]{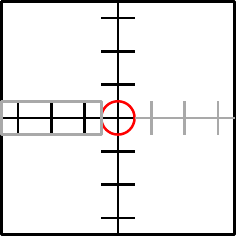}
}
\vskip .05in

\minip{.77}{12mm}{ However, if we cannot extend east (that is, if there is any elbow to the east of the cross circled red), then we must extend north individually for each of the crosses boxed in grey:
}
\hfill
\minip{.20}{0mm}{
  \includegraphics[scale=0.5,clip=true]{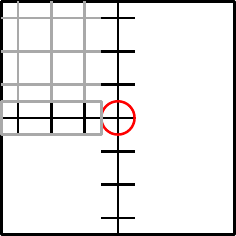}
}
\vskip .05in

In fact, we have shown something slightly stronger than the theorem. The configuration we started with here already had the cross circled in red contained within an entire column of crosses. Thus, what we showed is that if a cross is contained within an entire column or row of crosses, and the perpendicular direction has a line of crosses extending all the way to the wall in only one of the two possible directions, then it must be part of a northwest or southeast partition. 
\\

Now let us take a look at (d):

\begin{figure}[htbp] \centering
	\includegraphics[scale=0.5,clip=true]{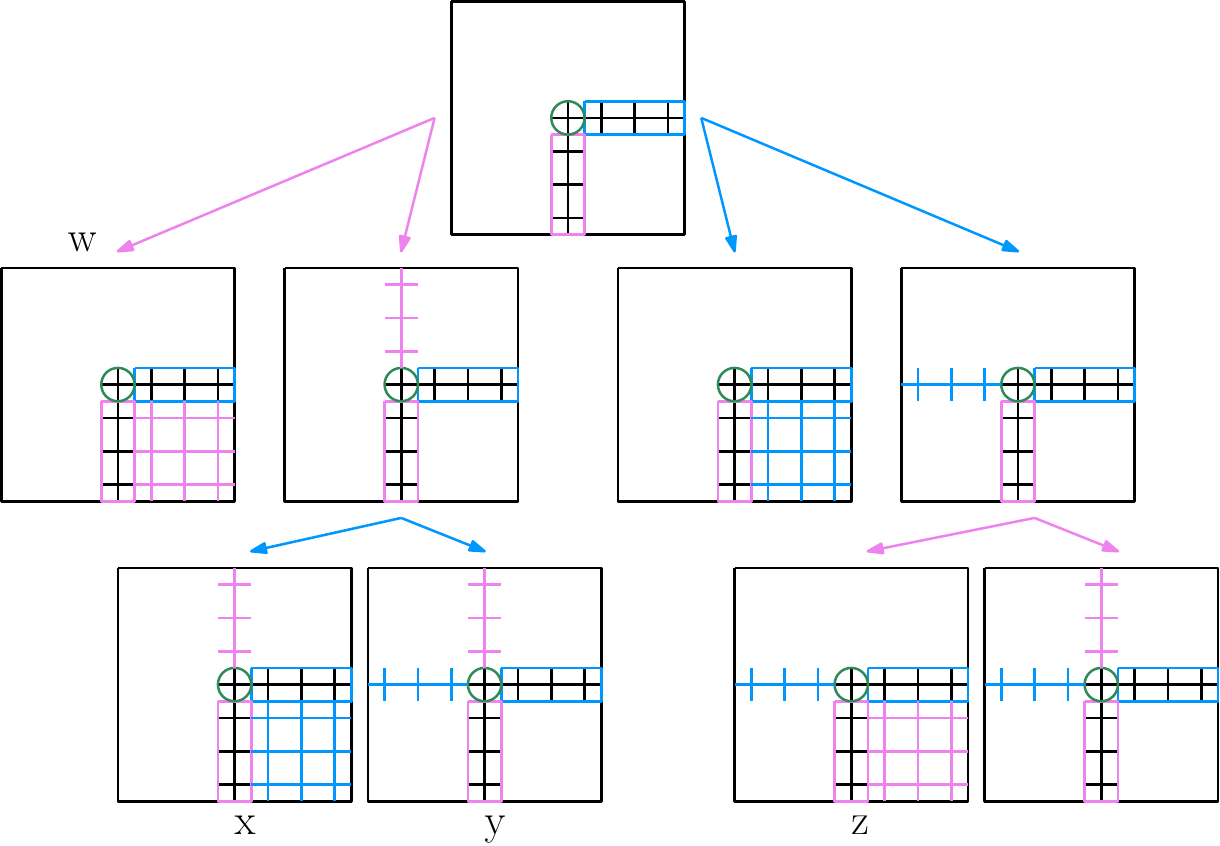}
\end{figure}

Here we have boxed the crosses in pink and light blue. The pink crosses satisfy the SW condition but not the NE. The light blue crosses satisfy the NE condition but not the SW. Following the arrows through the figure, where the pink arrows represent resolving the SW condition, and analogously for the light blue arrows (the details are the same as the description for the case of configuration (c) above), we obtain 6 final configurations. Some are redundant, so the unique ones are labelled (w), (x), (y), and (z). 
\\

It is clear that by symmetry, all other cases in (a)-(i) will just be certain reflections/rotations of what we have already shown. This proves the theorem. 

\end{proof}

\begin{proof}[Third proof of Proposition \ref{prop:delsmooth}]

We are trying to show that if the point $\lambda$ is smooth on $\Pi_f$, then $\lambda$ is also smooth at any deletion or contraction of $\Pi_f$. 

Since we assume that $\lambda$ is smooth on $\Pi_f$, we can apply Theorem \ref{thm:smoothone}, which says that the pipe dream for $\Pi_f$ on $U_\lambda$ is rigid. (This implies that any subset of the pipe dream (for example, any maximal rectangle) is also rigid, since no two pipes containing any cross and near-miss combo in the whole pipe dream implies the same thing for any subset; but we will not need this fact for this proof, and this can already by deduced from the Main Theorem since rectangles that reduce to NW/SE partitions are rigid). This also implies that if we do deletion or contraction by filling in entire rows or columns of crosses, as explained in Prop \ref{prop:delpipedream}, we do not need to (since in fact we cannot) do any moves prior to the filling in - rather, in this smooth/rigid case, deletion (contraction) of column (row) $i$ simply involves taking the $ith$ column (row) and brute-force changing every elbow tile in this column (row) into a cross tile. In the second (combinatorial) proof of Proposition \ref{prop:delsmooth} above, we did a detailed argument to show that this ``brute-force filling in" did not create any new moves, but now we show that we can bypass that argument by applying the Main Theorem.

Recall the definition of ``reducing to NW/SE partitions": this means having a rectangular shape where the northwest and southeast were partitions, and additionally, any number of rows or columns could be filled in with crosses; everything else is elbows. The Main Theorem says that (in this smooth case) all maximal rectangles reduce to NW/SE partitions; now, taking any one of these maximal rectangles (which reduce to NW/SE partitions) and filling in a row or column with crosses is still precisely a shape that reduces to NW/SE partitions. Therefore, it is still smooth. 

\end{proof}

\section{Atomic Positroid Pairs}

Next, we describe the singular pairs that are lowest in the order: recall that we called these ``atomic positroid pairs." By definition, any deletion or contraction of an atomic positroid pair results in a smooth positroid pair (see the figures in the introduction). The order we defined on positroid pairs has the property that for any given pair, as we go down in the order, we can only delete or contract a finite number of times. Because of this, any singular pair lies over an atomic pair, i.e. these minimal elements generate the order ideal.

\begin{Thm}\label{thm:atomicpd}

The atomic positroid pairs have affine pipe dreams that look like the following:

 \begin{figure}[htbp] \centering
	\includegraphics[scale=0.6,clip=true]{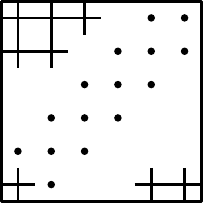}
\end{figure}

More precisely, the affine pipe dreams looks like the following, with each square labeled ``A" being the same, and taking the form of the previous figure:

 \begin{figure}[htbp] \centering
	\includegraphics[scale=0.6,clip=true]{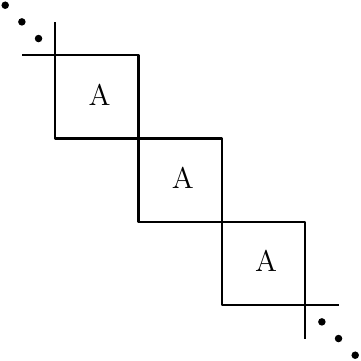}
\end{figure}

\newpage

Specifically:

(1) The atomic positroid pairs have pipe dreams in a square shape: so they are in $Gr(k,2k)$ and on points where $\lambda$ consists of $k$ consecutive columns. 

(2) In this pipe dream, there is a single cross along the southwest to northeast (longest) diagonal, with the diagonals directly adjacent to this longest diagonal being free of any crosses.

(3) Besides this one cross, there are partitions of crosses in the northwest and southeast corners, and no others.

\end{Thm}

\begin{proof}

We begin by proving that for an atomic positroid pair, in each of its pipe dreams there must be a single cross that can move. Suppose for contradiction that there exist two crosses that we call cross A and cross B, both with moves. Recall that by the definition of atomicity, this implies that any deletion or contraction must eliminate both crosses. We show that in general we can always eliminate one of the crosses while preserving the other one through a deletion or contraction - unless the pipe dream stretches infinitely in both the northeast and southwest directions, which gives us the contradiction. 
\\

In detail: suppose that cross A is at $(i_1,j_1)$ and cross B at $(i_2,j_2)$. The argument below will use the fact that, if a deletion
eliminates both crosses, then they have to be in the same column (the case of contraction: they have to be in the same row). Without loss of generality, at least the first or second coordinate must differ, or else cross A and cross B are the same cross, so let's take $j_1=j_2$: they are along the same column, and let's take $i_1$ to be above $i_2$ in the section of the pipe dream we are looking at. Now, clearly, deleting along this column eliminates both crosses, but since they are on different rows with cross A above cross B, we could contract cross A, which would leave cross B unchanged - that is, unless cross B has a move northeast to a new position $(i_1,j_3)$ along the same row as A but to its east. Now, we have solved the issue of contraction along $i_1$ eliminating both crosses, but the problem is that B is in a new easternmost column $j_3$ so we can delete $j_3$ to only eliminate B - that is, unless  A has a move northeast to a new position $(i_3,j_3)$ that is in the same column as B but north of B. But now we run into the issue that cross A has a new possible contraction at $i_3$. Thus, continuing on, we would have to do this infinitely in the northeast (as well as analogously in the southwest ) direction in order to preserve atomicity, which is impossible since our pipe dream shapes are bounded. 
\\

Now that we have proven that there is a single cross with a move, let us consider the possible positions of this cross: it must exist along the main diagonal of a square-shaped pipe dream; this is because, if this cross's moves missed any row or column, then we could delete this skipped row or column while leaving the cross (and its moves) intact, contradicting atomicity. Furthermore, this means that we have a clear diagonal along which this cross moves; in other words, the diagonals above and below this diagonal are free of any crosses, otherwise such a cross would hinder the free movement of this cross along the main diagonal (this relies on the fact just proven that no other cross can move): this proves (2). Since the pipe dream shape includes every row and column along a southwest to northeast diagonal, and no other rows/columns, and because in the shapes we are considering the non-convex portions only occur along the southwest and northeast corners, the boundaries of the pipe dreams shape must extend straight from the southeast and northeast corners: thus, the shape must be a rectangle. But since we already showed that every row and column must intersect the main diagonal, the shape in fact must be square. This proves (1). Since we already proved that this is the only cross that can move, everything else must be rigid, and this proves (3). 
\end{proof}

\begin{Prop}\label{prop:atom=sing}
A point $\lambda$ being singular on a positroid variety $\Pi_f$  is equivalent to the pipe dream for $\Pi_f$ on $U_\lambda$ being able to reach an atomic configuration via a series of deletions and contractions (that is, d/c-reduces to an atomic configuration). 
\end{Prop}

\begin{proof}

If $(\Pi_f,\lambda)$ d/c reduces to an atomic configuration, since atomic positroid pairs are singular by definition (minimally singular), and since we already proved in Prop \ref{prop:delsmooth} that the order on deletion and contraction is such that, for a singular element, everything higher up in the order is also singular, $(\Pi_f,\lambda)$ must be singular. 

In the other direction, $\Pi_f$ is finite dimensional, so a finite number of deletions and contractions will take it to the point $\lambda$, which is smooth in itself. Thus, every path going down in the order from a pair $(\Pi_f, \lambda)$ becomes smooth. Suppose $\Pi_f\in Gr(k,n)$. Then there are $k$ contractions (which can be applied to the columns in $\lambda$) and $n-k$ deletions (for the other columns, which cannot be contracted since $U_\lambda$ contains only the $T$-fixed point $\lambda$). Furthermore, suppose $\Pi_f$ is singular. Then these $n$ deletions and contractions are either (1) all smooth, in which case $(\Pi_f, \lambda)$ is atomic, or (2) at least one is singular, which let's call $(\Pi_{f'},\lambda')$. We can apply the same argument to $(\Pi_{f'},\lambda')$, which will have $n-1$ possible deletions and contractions - either it is atomic, or it has a deletion or contraction that remains singular. In this way, we can view each point in the ordering as spawning a tree of children below. By Prop \ref{prop:delsmooth}, every singular point only has singular parents, and so singular points are connected. Because all children become smooth eventually, these connected subtrees of singular points eventually end, with the last point being an atomic pair. See the next figure for an example.

\end{proof}

\begin{Exa}
Here, we provide an example of a part of the ordering. In particular, we show the positroid variety given by siteswap $f=342333$ containing point $\lambda=\{1,2,5\}$. This point is singular, which we have denoted by putting this positroid pair inside a blue rectangle. 

\newpage

 \begin{figure}[htbp]
	\includegraphics[scale=0.65,angle=-90,clip=true]{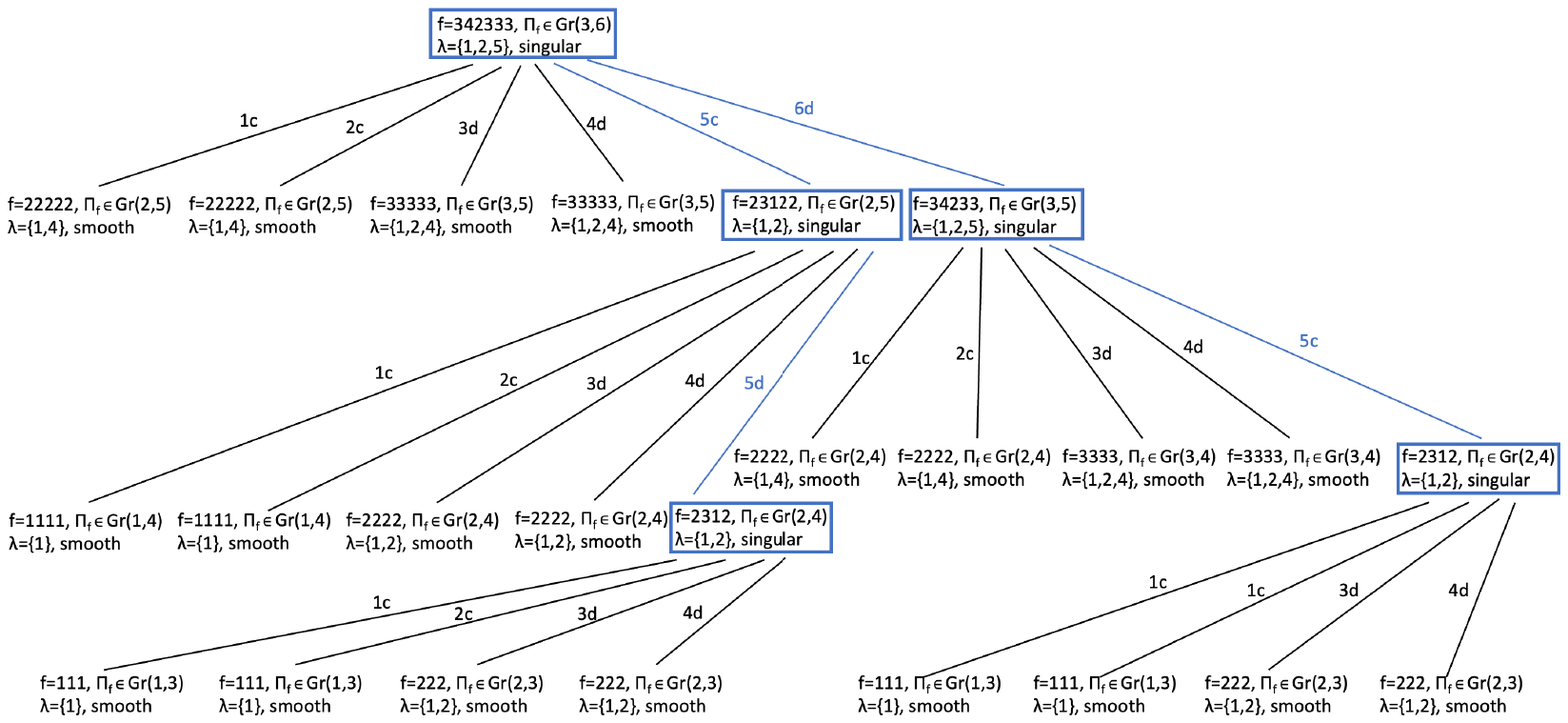}
\end{figure} 

We consider the set of all its deletions and contractions, denoted by the lines, and labeled for example, by ``2c" meaning that we contract the 2nd column or ``5d" meaning that we delete the 5th column. We stop doing further deletions and contractions once we reach a smooth point, since by Proposition \ref{prop:delsmooth}, anything lower in the order is still smooth. We see that the pair $(342333,\{1,2,5\})$ has exactly 2 deletions or contractions that are still singular, which are contraction of the 5th column or deletion of the 6th column. We can continue to delete or contract and obtain something singular until we get to $f=2312, \lambda=\{1,2\}$. Any further deletions or contractions result in a smooth point. Thus, by definition, this is an atomic pair (in fact, the only atomic pair that d\c reduces from $(342333,\{1,2,5\})$). And indeed, if we draw out the pipe dream, we see that $(2312, \lambda=\{1,2\})$ has a pipe dream that matches the form depicted in Theorem \ref{thm:atomicpd}:

 \begin{figure}[htbp]
	\includegraphics[scale=1,clip=true]{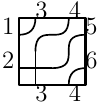}
\end{figure} 

\end{Exa}

We prove one more proposition in this section before making some final remarks.

\begin{Prop}\label{prop:atomrect}
Every atomic positroid pair comes from within a maximal rectangle. More precisely, if we have a pipe dream $\delta$ and there exists a series of deletions and contractions that take it (d/c reduces) to an atomic configuration that we call $\delta_{atomic}$, the process of deletion and contraction will have removed every tile outside of some maximal rectangle $R$ contained in $\delta$ (and most likely, some tiles within $R$ as well). 
\end{Prop}

\begin{proof}
This can be seen via an argument based on the following figure:

 \begin{figure}[htbp] \centering
	\includegraphics[scale=0.6,clip=true]{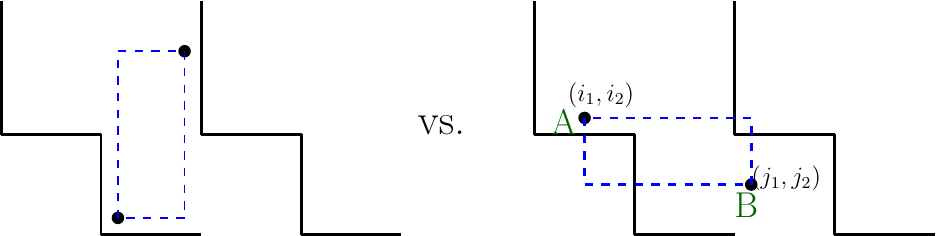}
\end{figure}

On the left side, we see that any pair of tiles that are southwest/northeast of each other must be contained within a maximal rectangle: this follows purely from the shapes we are considering. However, if they are southeast/northwest of each other, this might not be the case. In this case, if northwest tile A has coordinates $(i_1,i_2)$ and southeast tile B has coordinates $(j_1,j_2)$, then they are not contained in a maximal rectangle if either $(i_1,j_2)$ or $(i_2,j_1)$ is not contained in the shape. In such a case, there can be no series of deletions or contractions that put A and B into the same atomic pipe dream. The reason is that we cannot delete row $i_1$ or contract column $i_2$ without eliminating tile A, and we cannot delete row $j_1$ or contract column $j_2$ without eliminating tile B. Therefore, we cannot eliminate the spaces $(i_1,j_2)$ or $(i_2,j_1)$: they will always exist, so there will always be a convex hole, meaning that tiles A and B will never be able to fit into a square shape. This completes the proof of the fact that every atomic positroid pair comes from within a maximal rectangle.
\end{proof}

Let $\delta \in PD_{(\Pi_f,\lambda)}$ be an affine pipe dream for the positroid variety $\Pi_f$ on $U_\lambda$. We recap some of the statements that have been proven in this thesis.

\begin{equation*}
\begin{split} 
&\text{$\delta$ has a move within some maximal rectangle} \\
&\iff \text{$\delta$ has a move (is not rigid)} \\
&\iff \text{$\Pi_f$ is singular} \\
&\iff \text{$\delta$ d/c-reduces to an atomic pair} \\
&\iff \text{$\delta$ d/c-reduces to an atomic pair in some maximal rectangle}
\end{split} 
\end{equation*}

The first equivalence comes from Proposition \ref{prop:insidemaxrect} (one direction is trivial) or alternatively Corollary \ref{cor:box}, the second equivalence comes from Theorem \ref{thm:smoothone} (as well as Prop \ref{prop:BAPmoves}), the third equivalence is the content of Proposition \ref{prop:atom=sing}, and the last equivalence was just shown in Proposition \ref{prop:atomrect}. It might be hoped that, due to the apparent similarity of the first and last statements (both involving a condition within a maximal rectangle), that an equivalence between these two statements could be proved directly, which would produce a different proof of Proposition \ref{prop:insidemaxrect} using atomicity, and thereby yield an alternate proof of the Main Theorem \ref{thm:mainthm}. Unfortunately, the proof is more difficult than it seems. In one direction, going from $\delta$ having a move within some maximal rectangle to showing that $\delta$ d/c reduces to an atomic pair in that same maximal rectangle can be done by arguments involving deleting and contracting all rows and columns except for the positions where two pipes have a cross and near-misses (the argument could be made a bit easier by making use of Corollary \ref{cor:box}). The other direction is more difficult because the process of deletion or contraction occurs by doing all possible moves (in order to make sure that all pipe dreams for a given bounded permutation with the maximum number of crosses along the column/row to be deleted/contracted are obtained). This means that immediately upon invoking the notion of d/c-reduction, we are no longer just considering the single affine pipe dream $\delta$ but rather a possibly very large set of affine pipe dreams. In particular, we can no longer focus in on a single maximal rectangle (e.g. the maximal rectangle inside which we hope the atomic pair arises or the maximal rectangle where we hope to prove there exists a move) because we have to consider a set that includes pipe dreams where a cross-elbow move has shifted to being outside that single maximal rectangle we originally intended to focus on.

\section{Skew Partitions and Schubert Varieties}

A \textbf{skew partition} can be denoted $\nu \backslash \rho$ for two partitions $\nu$ and $\rho$ where $\rho$ is completely contained in $\nu$: see the picture below. It will be easier to think of the partition as being rotated by 45 degrees, so the pipes enter along the bottom and exit through the top:

 \begin{figure}[htbp]
	\includegraphics[scale=0.5,clip=true]{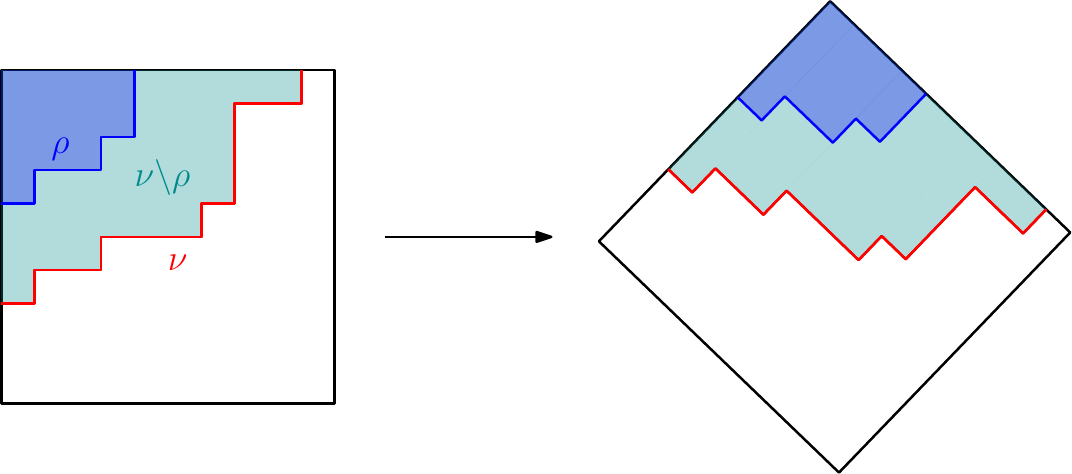}
\end{figure}

\begin{Prop}\label{prop:nonaffineword}
Let $\pi \in S_n$. The following are equivalent:

(1) $\pi$ is 321-avoiding.

(2) $T \curvearrowright X_0^\pi$ contains dilation.

(3) $\pi$ has a unique reduced word up to commuting moves (a unique heap).

(4) The heap is a skew partition.
\end{Prop}

\begin{proof}

We prove the equivalence of (1) with (2), (3) and (4). 

First, we show $(1) \Leftrightarrow (2)$. Recall that we called (2) the condition of being cominuscule, and this is the property that there exists a circle inside the torus that acts by dilation. This means that this circle acts with all weights one. Since the roots of $X_0^\pi$ correspond to all inversions in $\pi$ (that is, a pair $1\leq i<j \leq n$ such that $\pi(i)>\pi(j)$ implies a root $x_j - x_i$), this implies that we can simultaneously set all $x_j - x_i=1$, for all  $1\leq i<j \leq n$ such that $\pi(i)>\pi(j)$.
Note that, dually, there is a map from the weight lattice of the torus to the weight lattice of the circle, which is the integers. Since we want all our roots to go to 1, this means that all the roots have to lie on a hyperplane, specifically the hyperplane of things that go to 1 under this dual map. 

For the $(\Rightarrow)$ direction, if $\pi$ is not 321-avoiding, then there exists $i<j<k$ such that $\pi(i)>\pi(j)>\pi(k)$. This would imply that we would need to find $x_i, x_j, x_k$ satisfying the 3 equations: $x_j-x_i=1$,  $x_k-x_j=1$,  $x_k-x_i=1$, and clearly this is inconsistent, since if we add the first two equations we get $x_k-x_i=2$.

For the $(\Leftarrow)$ direction, suppose $\pi$ is 321-avoiding. This means that for any $1\leq k\leq n$, if we find a $j<k$ such that $\pi(j)>\pi(k)$, then there will be no $i<j$ such that $\pi(i)>\pi(j)$. Therefore, we can set $x_k =1$ and $x_j=0$ so that $x_k-x_j=1$. Note that $k$ cannot descend to the right (that is, we will be unable to find a $k<l$ such that $\pi(k)>\pi(l)$, because then $j,k,l$ would form a 321-triple), and similarly $j$ cannot ascend to the left (since as we already noted, $i<j$ such that $\pi(i)>\pi(j)$ would result in $i,j,k$ forming a 321-triple); this proves that we can assign $0$'s and $1$'s to all $x_i$ to satisfy all roots equalling $1$ consistently, since $x_i$ participates in each descent only as the higher or lower number, not both. 

None of these arguments relied on the permutation being finite, so the equivalence of (1) and (2) also holds in the affine case $\hat{S}_n$.

Next, we show $(1) \Leftrightarrow (3)$. First, note that (3) is equivalent to the statement that there are no braid moves in any reduced word for $\pi$, where the braid moves is defined as $s_i s_{i+1} s_i = s_{i+1} s_i s_{i+1}$. This follows from Tits' Theorem, that any word for a permutation can be obtained from any other word via (1) commuting moves or (2) braid moves. Since we allow commuting moves, the statement of (3) is that there are no braid moves. Second, note that (1) is equivalent to there not existing any pattern of the form pictured below:
 \begin{figure}[htbp]
	\includegraphics[scale=0.5,clip=true]{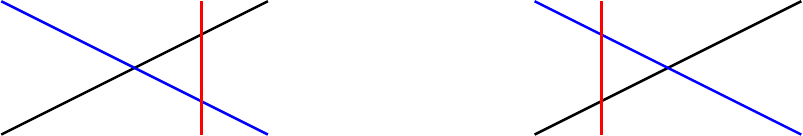}
\end{figure}
The reason is because if $\pi$ is not 321-avoiding, then it must contain a set of 3 pipes $i,j,k$ such that $\pi(i)>\pi(j)>\pi(k)$. Another way to state this pattern is that there is a set of 3 pipes such that each one intersects the other two. 

For $(\Rightarrow)$, if (3) is not true, then it contains a braid move, so its reduced word (up to commuting moves) allows for a braid move; that is, there is a combination $s_i s_{i+1} s_i$ or $s_{i+1} s_i s_{i+1}$. These two patterns correspond to the figure above, and since reducedness implies that double-crossings are not allowed so the pipes cannot uncross (stay in this position relative to each other), the permutation is not 321-avoiding. 

For $(\Leftarrow)$, we noted above that (1) is equivalent to not being able to find a set of three pipes such that any one intersects the other two, as in the figure above. Note that if we do spot such a pattern, we can search for the smallest one; in other words, we can make sure that there are no pipes that, for example in the figure on the left, intersect the blue and black pipes to the right of their intersection point and to the left of the red pipe, or another pipe northeast of the blue pipe and below the red-black intersection. This implies that, for this set of 3 pipes forming the smallest triangle, there will be no other pipes entering the triangle, so that in their heap diagram, they will be read as $s_{i+1} s_i s_{i+1}$ if it looks like the first diagram, and $s_i s_{i+1} s_i$ if it looks like the second diagram. 

Next, we show $(1) \Leftrightarrow (4)$

For $(\Rightarrow)$, we assume that $\pi$ is 321-avoiding. Assume for contradiction that the heap is not a skew partition. Consider the following figure:

 \begin{figure}[htbp]
	\includegraphics[scale=0.6,clip=true]{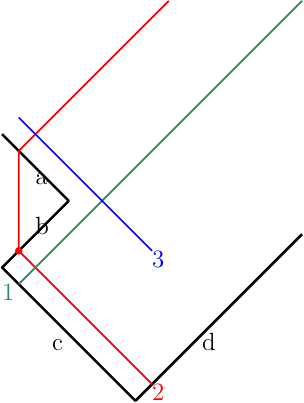}
\end{figure}

\newpage

The fact that it is not skew implies that there is an inward, nonconvex opening, either along the left or the right. There must be a first such opening (this would be a place where a bottom boundary of form (1) NW-SE touches a boundary of form (2) SW-NE such that the NW corner of (1) joins the SW corner of (2)). We displayed this opening with the 2 boundaries labeled $a$ and $b$. The (1) NW-SE boundary is labeled $c$, and since of course the non-skew partition must close up, $c$ must touch another boundary of form (2) SW-NE on its right, which we have labelled $d$. Along the edge $b$, there must be a first pipe entering (hitting the red dot), and we have colored this pipe in red and labelled it $2$. After it hits $a$, it must continue northeast until it hits the top boundary. There must be a first pipe emerging from $c$, which we have colored green and labeled $1$. It must also continue northeast, parallel to $2$, until it hits the top boundary. Finally, perpendicular to pipes $1$ and $2$ must be a pipe $3$ that we have labelled in blue. Clearly, in order to cross them perpendicularly, it must start later than both $1$ and $2$, and since it crosses perpendicularly moving left, it must end earlier than both $1$ and $2$. Therefore, pipes 1, 2, and 3 form a 321 triplet. Since these arguments were general, only assuming the most general property of being non-skew, this proves that $\pi$ is not 321-avoiding. 

For $(\Leftarrow)$, we show that skew partitions are 321-avoiding. We noted above that being 321-avoiding is equivalent to not having the diagram above where every one of 3 pipes intersects the other 2. 
Assume that we have a skew partition. Consider pipes coming in along the bottom edge. They enter in 2 types of ways: along diagonal lines that go (1) NW-SE, and diagonal lines that go (2) SW-NE:

 \begin{figure}[htbp]
	\includegraphics[scale=0.5,clip=true]{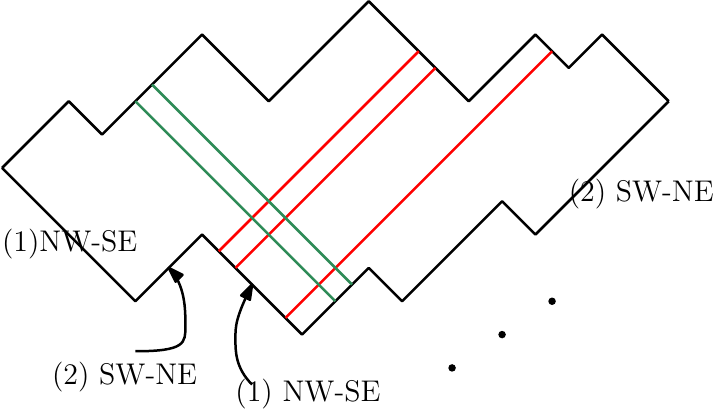}
\end{figure}

In the figure, we have also displayed with some example pipes, that the (red) pipes coming from the (1) NW-SE bottom boundary move northeast, and the the (red) pipes coming from the (1) SW-NE bottom boundary moves northwest. All pipes coming from (1) NW-SE are parallel, and never cross. Additionally, reading from left to right, they are in the same order along the bottom as the top. The same is true for the (2) SW-NE pipes. Thus, in a skew partition being all pipes can be divided into 2 classes, each of which consists of lines parallel to all other pipes within its class. Let's call the classes $A$ and $B$. All pipes in class $A$ are parallel, so can only intersect pipes of class $B$. All pipes in class $B$ are parallel, so can only intersect pipes in class $A$. Let's pick an arbitrary pipe $a$, which, without loss of generality, is in class $A$. If $a$ intersects a pipe, this intersected pipe must be a pipe $b$ in class $B$. Now, in order to violate the 321-avoiding condition, we must have a pipe that intersects both $a$ and $b$. However, every pipe must be in class $A$ or class $B$. If it is in class $A$, then it can only intersect $b$ and not $a$; if it is in class $B$ then it can only intersect $a$ and not $b$. Thus, we have a contradiction. Note that this direction generalizes easily to the affine case as well.

\end{proof}

\begin{Prop}\label{prop:flagscorollary}
Let $Fl(n)$ denote the variety of flags in $\C^n$. Let $w\in S_n$ denote a permutation on $n$ elements so that $X_w$ denotes a Schubert variety in this flag manifold, and let $v\in S_n$ be 321-avoiding with $v\geq w$ in Bruhat order. Then $X_w$ is smooth at $v$,  if and only if there exists a pipe dream for $w$ inside $v$'s skew partition (defined above Prop \ref{prop:nonaffineword}) such that all maximal rectangles reduce to NW/SE partitions. (In this case, the pipe dream will be unique.) 
\end{Prop}

\begin{proof}
We assumed that $v\in S_n$ is 321-avoiding, so we can use Proposition \ref{prop:nonaffineword}. By (3) of this proposition, $v$ has a unique heap, and by (4) of the same proposition, this heap is a skew partition, so $v=\nu/ \rho$ for some partitions $\nu$ and $\rho$. Then $w$ be given by a subword inside $\nu / \rho$. We can furthermore apply equation (*) in the beginning part of Subsection \ref{subsect:smooththm} on the Smoothness Theorem; part (1) is the only part that deals with positroid varieties in the Grassmannian, so parts (2), (3) and (4) of equation (*) relate the T-equivariant cohomology $[X_{w} \cap X_\circ^{v} \subseteq X_\circ^{v}]$ inside the affine flag variety to the number of pipe dreams via the AJS/Billey formula. This allows us to take advantage of the combinatorial content of the Main Theorem \ref{thm:mainthm}, specifically Proposition \ref{prop:insidemaxrect} that all moves exist within maximal rectangles and Theorem \ref{thm:NW/SE} that all rigid rectangles reduce to NW/SE partitions. Therefore, if all maximal rectangle within the pipe dream for $w$ inside $\nu/ \rho$ reduce to NW/SE partitions, then the pipe dream for $w$ inside $\nu/\rho$ is rigid.

Getting the multiplicity as the coefficient of $h^{\codim X}$ required the use of Theorem \ref{thm:Rossmann} from Rossmann, which required the existence of a circle within $T$ acting by dilation (cominuscule), and this is exactly the content of Proposition \ref{prop:nonaffineword} part (2). Therefore, the algebraic geometry statement of the Main Theorem applies as well; in other words, we can conclude that if all maximal rectangles for the pipe dream of $w$ inside $\nu/ \rho$ reduce to NW/SE partitions, then $X_w$ is smooth at $v$.

\end{proof}

One then wonders whether the corresponding statement is true in more generality in the affine case as well. Unfortunately, it is not true. The issue is that it is possible for two pipes $A$ and $B$ that cross and near-miss combination on a cylinder to be such that, when drawn as an infinitely long strip on a page, the near-miss for example involves pipes $A$ and $C$ instead of pipe $B$. In particular, this implies that Proposition \ref{prop:insidemaxrect} no longer holds. Such a counterexample requires a tile from which we can move both south and west, as well as north and east (in perpendicular direction) and reach the same letter $s_i$ on the same square; this occurs for the case of $n$ even when there is a tile from which one can move an additional $\frac{n}{2}$ tiles in the NE and NW directions or SE and SW directions (or $\frac{n-1}{2}$ and $\frac{n+1}{2}$ in the odd case). However, in the case of \cite{Sni}, the shapes are bounded of height $k$ and width $n-k$, so the maximum a tile could move would be $k-1$ squares NW and $n-k-1$ squares NE. Thus, there is always a rectangle containing any cross and near-miss combination of pipes. We hope to comment more about the affine case in a companion paper \cite{Flu} to be released soon, which will focus more on computations and examples, as well as connections to several other notions in the literature.

\section{Appendix A}

In this appendix, we include the details of how to calculate the siteswap of $del_i(\Pi_f)$, the positroid variety obtained from $\Pi_f$ by deleting the $i$-th column. \cite{SOh} already gave an algorithm for doing deletion and contraction on decorated permutations and Grassmann necklaces; here we describe a pictorial way to do this for siteswaps. 
\\

It is perhaps easiest to see what is going on diagrammatically:

 \begin{figure}[htbp]
	\includegraphics[scale=0.6,clip=true]{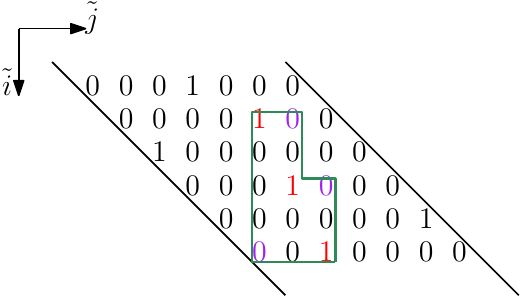}
\end{figure}

Or, step by step:

 \begin{figure}[htbp]
	\includegraphics[scale=0.6,clip=true]{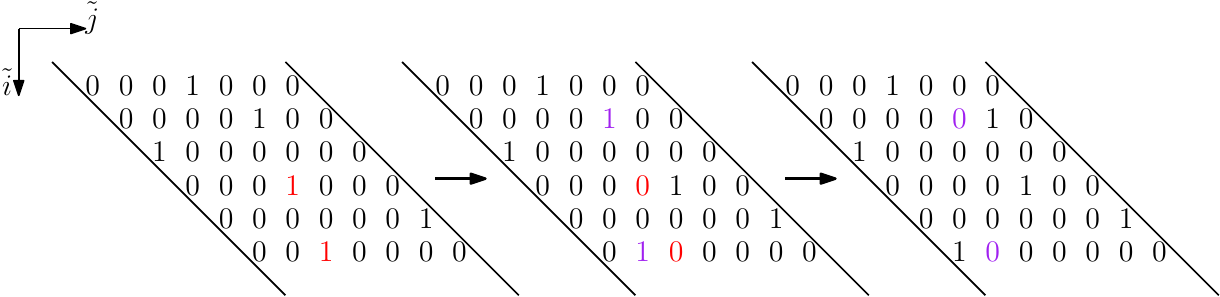}
\end{figure}

The L-shaped box on top represents a successive process of inversions that move the $1$ in the $i$th row left, until it lands on the left diagonal line (which is the $(i,i)$ position in the $\tilde{i}$-$\tilde{j}$ coordinate system). Thus, the last inversion performed will involve the $1$ that is directly north of this $(i,i)$ spot to bring it south. The case of contraction is analogous, except that the goal is to move the $1$ in the $i$th row rightwards, until it lands on the $(i,i+n)$ position. At every step, we look at the previous row for deletion (or next row for contraction), and see if there is a $1$ lying between the value $i$ on the $\tilde{j}$ axis, and where the $1$ on the $i$-th row currently is; if there is, we do the inversion. This is simply the statement that we are doing the minimal number of weak Bruhat moves to get from our original affine permutation to one that has a $0$ in the $i$-th slot. This results in the largest positroid variety contained in the original one that has a $0$ in the $i$-th slot.  
\\

Here is a more algorithmic way of writing what we have just described in words: 

Let $g$ denote the bounded affine permutation, and $f$ its corresponding siteswap. For deletion, start with $i$ and $g(i)$, we want $g(i)=i$ since this means $f(i)=0$. Test $i-1$, $i-2$,... successively and if necessary perform the following operation until the goal $g(i)=i$ is reached. So we begin by testing $i-1$: if $i<=g(i-1)<g(i)$, then transpose $i$ and $i-1$: define the new $g(i):=g(i-1)$ and $g(i-1):=g(i)$. (Whenever such an inversion is performed, $g(i)$ is a little closer to $i$.) Continue with $i-2$ (i.e. the new condition to be tested is: $i<=g(i-2)<g(i)$) and so on until $g(i)=i$ is reached. 
\\

For contraction, start with $i$ and $g(i)$, we want $g(i)=i+n$ since this means $f(i)=n$. Test $i+1$, $i+2$,... successively and if necessary perform the following operation until the goal $g(i)=i+n$ is reached. So we begin by testing $i+1$: if $g(i)<g(i+1)<i+n$, then transpose $i$ and $i+1$: define the new $g(i):=g(i+1)$ and $g(i+1):=g(i)$. (Whenever such an inversion is performed, $g(i)$ is a little closer to $i+n$.) Continue with $i+2$ (i.e. the new condition to be tested is: $g(i)<g(i+2)<i+n$) and so on until $g(i)=i+n$ is reached. 
\\

The proof that this is the correct process follows right from the proof of Proposition \ref{prop:del2def}. There, we showed that the deletion of the $i$th column of a positroid variety $\Pi_f$ is a single positroid variety $\Pi_{f'}$, so is described by a siteswap: $\Pi_{f'}=\Pi_{del_i(f)}$. We showed there that this $\Pi_{del_i(f)}$ is the largest positroid variety $\Pi_{f'}$ contained in $\Pi_f$ such that $f'(i)=0$. We know that for $\Pi_{f'}$ to be a positroid variety contained in $\Pi_f$ is equivalent to $f'$ being reachable by a series of Bruhat moves from $f$, and each such weak Bruhat move decreases the dimension of the positroid variety by 1. Since we want the largest such positroid variety, we want the minimal number of such moves. This is exactly what the diagram in this Appendix does: it depicts pictorially the minimal number of Bruhat moves to get $f'(i)=0$; thus, it produces the largest positroid variety contained in $\Pi_f$ with $f'(i)=0$.

\section{Appendix B}

Here we give a self-contained introduction to affine pipe dreams. We do an explicit example to illustrate the process of taking in a pair $(\Pi_f, U_\lambda)$ and associating a reduced subword to it. Also, in the introduction, we noted that one way to test whether a positroid variety has a unique pipe dream on an open patch (and thus is smooth there) is to draw the top and bottom pipe dream for that positroid variety on that open patch. If they are the same, then the pipe dream is unique so the positroid variety is smooth there. If the pipe dreams are different, then the positroid variety is singular there. \cite{Sni} notes in Section 4.3.1 that the top pipe dream comes from the lexicographically first reduced word, and the bottom pipe dream comes from the lexicographically last word. We noted that the details would be included in this appendix. 
\\

Let's take an example of a siteswap $2225377$ (bounded affine permutation $3,4,5,9,8,13,14$) inside Gr(3,7), and let's test it on $\lambda=\{2,5,7\}$. 
\\

Consider the following figure:

 \begin{figure}[htbp]
	\includegraphics[scale=0.7,clip=true]{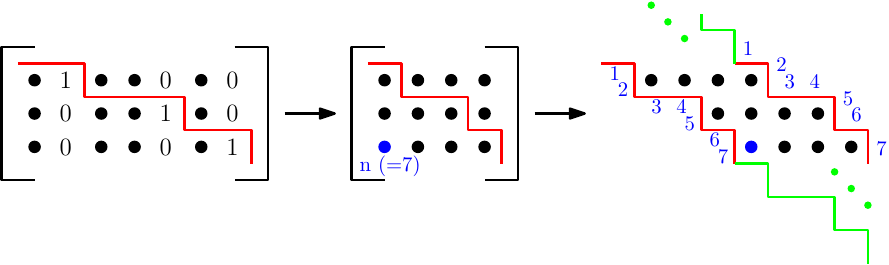}
\end{figure}

(1) In the first figure on the very left, the distinguished path (illustrated in red) is formed by moving horizontally starting from the northwest corner of the matrix until a $1$ in one of the identity columns is encountered, at which point the distinguished path moves down. Thus, the distinguished borders the north and east of each $1$ and otherwise is horizontal. \\
(2) In the middle figure, we collapse the identity columns to obtain the distinguished path. The tiles will be filled by transpositions $s_i$. Here, we have highlighted in blue the southwesternmost square. This will be $s_n$ (in the depicted example, $s_n = s_7$). \\
(3) Finally, we take one copy of the distinguished path and move it to the right (depicted also in red in the figure on the very right). We can continue adjoining distinguished paths (depicted in green) to the northwest and southeast of these two red distinguished paths, forming an infinite strip, which is the boundary of the affine pipe dream. Note the labelling along the boundary: it simply starts with ``$1$" at the top line segment of the distinguished path and increases moving downwards. Thus, we see why the blue dot will be the transposition $s_7$ (or just ``$7$" in the next figure): a cross where the blue dot is will switch the $7$-th and $8=1(\bmod 7)$-th pipes coming from the bottom boundary.

Filling in the rest of the tiles, we get: 

 \begin{figure}[htbp]
	\includegraphics[scale=0.8,clip=true]{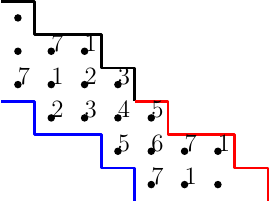}
\end{figure}

If we just look at one block, denoting the transpositions by $s_i$ again, it is:
\\

 \begin{figure}[htbp]
	\includegraphics[scale=0.7,clip=true]{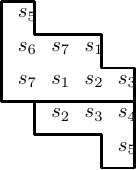}
\end{figure}

To get the bottom pipe dream, we have to pick a way of reading this. There are many equivalent ways of reading it, corresponding to the southwest rule in the proof of Prop \ref{prop:AJSrowcolform}, but if we want to read consistently row-by-row or column-by-column in each block, the choices are depicted below: 
\\

 \begin{figure}[htbp]
	\includegraphics[scale=0.7,clip=true]{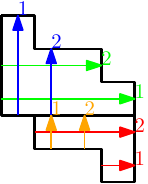}
\end{figure}

In other words, if we choose the direction of red followed by green, then the ordering will be $s_5, s_2 s_3 s_4, s_7 s_1 s_2 s_3, s_6 s_7 s_1, s_5$ (note that this is the standard choice for the word $Q_\lambda$ given in \cite{Sni}). If we apply this word to $(k+1,k+2,...k+n)$, then we get the bounded affine permutation corresponding to $+n$ for the $\lambda$ columns $2,5,7$ and $+0$ for all other columns; explicitly: $Q_\lambda(4,5,6,7,8,9,10)=(0,7,0,0,7,0,7)$. This shows that this pipe dream shape represents the neighborhood $U_\lambda$.
\\

Here's how reading it this way produces a word for siteswap $2225377$ (bounded affine permutation $3,4,5,9,8,13,14$):
\\

\begin{center}
\begin{tabular}{c|c|c|c|c|c|c|c}
1 & 2 & 3 & 4 & 5 & 6 & 7 & 8(=1) \\ \hline
3 & 4 & 5 & 9 & 8 & 13 & 14
\end{tabular}
\end{center}

Since the big word is $s_5, s_2 s_3 s_4, s_7 s_1 s_2 s_3, s_6 s_7 s_1, s_5$, we find the first $s_i$ (starting from the left with $s_5$) where there is a transposition with $f(i)>f(i+1)$. Then we can apply $s_i$ to switch $f(i)$ and $f(i+1)$. Here, the first one is $s_4$ which switches the $98$ to $89$:
\\

\begin{center}
\begin{tabular}{c|c|c|c|c|c|c|c}
1 & 2 & 3 & 4 & 5 & 6 & 7 & 8(=1) \\ \hline
3 & 4 & 5 & 9 & 8 & 13 & 14 & \\ \hline
3 & 4 & 5 & 8 & 9 & 13 & 14 & 10 \\
\end{tabular}
\end{center}

We continue on until we get the most general permutation:
\\

\begin{center}
\begin{tabular}{c|c|c|c|c|c|c|c}
1 & 2 & 3 & 4 & 5 & 6 & 7 & 8(=1) \\ \hline
3 & 4 & 5 & 9 & 8 & 13 & 14 & \\ \hline
3 & 4 & 5 & 8 & 9 & 13 & 14 & 10 \\ \hline
7 & 4 & 5 & 8 & 9 & 13 & 10 & 14 \\ \hline
4 & 7 & 5 & 8 & 9 & 13 & 10 & 11 \\ \hline
4 & 5 & 7 & 8 & 9 & 13 & 10 & 11 \\ \hline
4 & 5 & 7 & 8 & 9 & 10 & 13 & 11 \\ \hline
6 & 5 & 7 & 8 & 9 & 10 & 11 & 13 \\ \hline
5 & 6 & 7 & 8 & 9 & 10 & 11 & 12 \\
\end{tabular}
\end{center}

Keeping track of which transpositions we used, we get:

 \begin{figure}[htbp]
	\includegraphics[scale=0.7,clip=true]{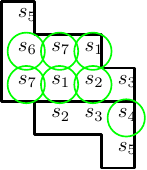}
\end{figure}

If we put crosses where the circles are, and elbows elsewhere, then we get the bottom pipe dream.
\\

In contrast, the top pipe dream corresponds to the pipe dream that we would obtain if we applied the same process as above, but where we make a few changes: \\
(1) We must use the inverse permutation. We write this as the numbers $g^{-1}(i) +n $.\\ 
(2a) We could change change the $s_i$'s such that they correspond to the labelling along the top boundary of the pipe dream shape rather than the bottom (this is depicted in the bottom row of the following figure). We will then read the word in the opposite ordering to the one above; in fact, any ordering obeying a northwest rule rather than a southwest rule (see the proof of Prop \ref{prop:AJSrowcolform}) will work. If we want to read in a consistent direction within each block (above or below the distinguished path), the choices are shown in the colored arrows in the next figure. \\
(2b) Alternatively to (2a) (but essentially equivalent), we could use Grassmannian duality to reflect the pipe dream shape across the northwest-southeast line, before going through with the same process of finding the bottom pipe dream; obviously, the changes of the $s_i$'s to match the (originally) top boundary happens automatically here. Finally, after doing this, we have to reflect back across the northwest-southeast line to get the top pipe dream.  

 \begin{figure}[htbp]
	\includegraphics[scale=0.7,clip=true]{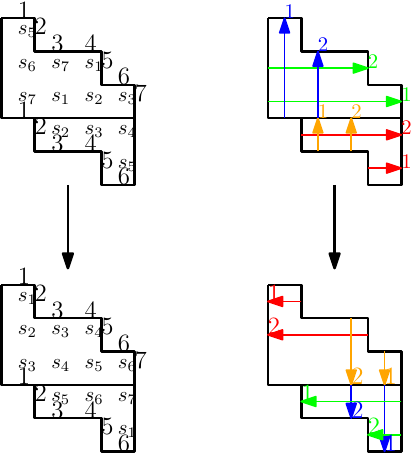}
\end{figure}

\section{Acknowledgements}
The author would like to thank his advisor Allen Knutson for many helpful discussions and for suggesting the approach.

\bibliographystyle{alpha}

\begin{thebibliography}{10}

  \bibitem[AJS/Billey]{AJS/Billey} S. Billey, 
Kostant polynomials and the cohomology ring for G/B, 
Duke Math. J. 96 (1999), no. 1, 205–224.

\bibitem[BB93]{BB93} 
N. Bergeron and S. Billey, 
RC-graphs and Schubert polynomials,
Experimental Math. 2 (1993), no. 4, 257–269.

\bibitem[BilBra03]{BilBra03}
S. Billey and Tom Braden,
Lower Bounds for Kazhdan-Lusztig Polynomials from Patterns,
{\tt math.RT/0202252}

\bibitem[BilWar]{BilWar}
S. Billey and G. Warrington,
Kazhdan-Lusztig Polynomials for 321-Hexagon-Avoiding Permutations,
{\tt math.CO/0005052}

\bibitem[BJN]{BJN}
R. Biagioli, F. Jouhet and P. Nadeau,
321-avoiding affine permutations and their many heaps,
Journal of Combinatorial Theory, Series A, Volume 162, February 2019, Pages 271-305.

\bibitem[BW22]{BW22} 
S. Billey and J. Weaver, 
A Pattern Avoidance Characterization for Smoothness of Positroid Varieties, 
Séminaire Lotharingien de Combinatoire (2022),
{\tt math.CO/2204.09013} 

\bibitem[Flu]{Flu} 
J. Fluegemann,
Smooth Points on Positroid Varieties Part 2,
to appear on the Arxiv soon.

\bibitem[GraKre]{GraKre} 
W. Graham and V. Kreiman,
Cominuscule Points and Schubert Varieties,
{\tt math.AG/1701.05956}

\bibitem[Hartshorne]{Hartshorne} 
R. Hartshorne,
Algebraic Geometry.

\bibitem[Kleiman]{Kleiman}
S. Kleiman,
The transversality of a general translate,
Compositio Mathematica, 28: 287–297, MR 0360616, 1974

\bibitem[KLS11]{KLS} A. Knutson, T. Lam, and D. Speyer,
Positroid Varieties: Juggling and Geometry,
  {\tt math.AG/1111.3660}
  
    \bibitem[KM03]{KM03} A. Knutson and E. Miller,
Subword Complexes in Coxeter Groups 
  {\tt math.CO/0309259}

\bibitem[KnScubDeg]{KnScubDeg}
A. Knutson,
Schubert patches degenerate to subword complexes,
{\tt  	math.AG/0801.4114}

\bibitem[KnTao01]{KnTao01} 
A. Knutson and T. Tao, 
Puzzles and (equivariant) cohomology of Grassmannians,
{\tt math/0112150}

\bibitem[Kumar]{Kumar}
S. Kumar, 
Kac-Moody Groups, their Flag Varieties and Representation Theory,
Springer Science Progress in Mathematics (PM, volume 204) 2002. 

\bibitem[LS90]{LS90} 
V. Lakshmibai and B. Sandhya, 
Criterion for smoothness of Schubert varieties in SL(n)/B, 
Proc. Indian Acad. Sci. (Math Sci.) 100 (1990), no. 1, 45–52.

\bibitem[Post]{Post} 
A. Postnikov, 
Total positivity, Grassmannians, and networks, 
{\tt math.CO/0609764}

\bibitem[Rossmann]{Rossmann} 
W. Rossmann, 
Equivariant multiplicities on complex varieties,
Astérisque, tome 173-174 (1989), p. 313-330.

\bibitem[SOh]{SOh}
S. Oh, 
Contraction and restriction of positroids in terms of decorated permutations, 
{\tt math.CO/0804.0882}

\bibitem[Sni11]{Sni} M. Snider,
Affine Patches on Positroid Varieties and Affine Pipe Dreams (Thesis)
  {\tt math.CO/1011.3705}
  
\bibitem[Stem]{Stem}
J. Stembridge,
Some Combinatorial Aspects of Reduced Words in Finite Coxeter Groups,
Transactions of the American Mathematical Society, Volume 349, Number 4, April 1997, Pages 1285–1332.

\bibitem[Thesis]{MyThesis} 
J. Fluegemann, 
Smooth Points on Positroid Varieties: A Dissertation Presented to the Faculty of the Graduate School of Cornell University, 
August 2024.


\end{thebibliography}

\end{document}